\documentclass[12pt, reqno]{amsart}
\usepackage{enumerate}
\usepackage[ocgcolorlinks]{hyperref}
\hypersetup{
  pdfstartview=FitH,
  bookmarksopen=false,
  pdfencoding=auto,
  colorlinks=false,
  linkcolor=blue,
  pdfinfo={
    Subject={35J60, 35B99},
    Keywords={polyharmonic maps, blow-up analysis, energy identity, no neck}%
  }
}
\usepackage[backrefs, msc-links, lite, abbrev]{amsrefs} 
\numberwithin{equation}{section}

\usepackage{aliascnt}
\newcommand{\mytheorem}[3]{%
  \newaliascnt{#1}{#3}
  \newtheorem{#1}[#1]{#2}
  \aliascntresetthe{#1}
  \expandafter\newcommand\csname #1autorefname\endcsname{#2}
}
\usepackage[%
  a4paper, 
  margin=1in, 
  marginparwidth=2cm, 
  marginparsep=3mm, 
  heightrounded, 
  includeheadfoot%
]{geometry}
\usepackage{graphicx}
\usepackage{etoolbox}
\usepackage{ifxetex}
\usepackage{enumitem}

\makeatletter
\newif\ifuserdefn@presubmit
\ifuserdefn@presubmit
\usepackage[notref, notcite]{showkeys} %final
\AtBeginDocument{\geometry{top=0in, bottom=0in}}
\fi
\makeatother

\usepackage[british]{babel}
\usepackage{palatino}
\mytheorem{mcor}{Corollary}{mthm}
\mytheorem{mlem}{Lemma}{mthm}

\newtheorem{defn}{Definition}[section]

\mytheorem{thm}{Theorem}{defn}
\mytheorem{lem}{Lemma}{defn}
\mytheorem{cor}{Corollary}{defn}
\mytheorem{prop}{Proposition}{defn}
\newtheorem*{rmk}{Remark}

\newtheorem*{claim}{Claim}
\newtheorem{nclaim}{Claim} %number claim

\newcommand{\I}{\mathrm{I}}

\newcommand{\RB}{\mathbb{R}}

\newcommand{\eps}{\varepsilon}

\def\Xint#1{\mathchoice
  {\XXint\displaystyle\textstyle{#1}}%
  {\XXint\textstyle\scriptstyle{#1}}%
  {\XXint\scriptstyle\scriptscriptstyle{#1}}%
  {\XXint\scriptscriptstyle\scriptscriptstyle{#1}}%
\!\int}
\def\XXint#1#2#3{{\setbox0=\hbox{$#1{#2#3}{\int}$ }
\vcenter{\hbox{$#2#3$ }}\kern-.6\wd0}}

\def\dashint{\Xint-}
\newcommand{\embedto}{\hookrightarrow}
\newcommand{\eqdef}{\mathpunct{:}=}

\newcommand{\xtext}[1]{\ensuremath{\,\text{#1}\,}}
\newcommand{\inner}[1]{\left\langle#1\right\rangle}
\DeclareMathOperator{\osc}{osc}
\newcommand{\pt}{\partial}
\newcommand{\tr}{\operatorname{tr}}
\newcommand{\abs}[1]{\left|#1\right|}
\newcommand{\set}[1]{\left\{#1\right\}}
\newcommand{\sbr}[1]{\left(#1\right)}
\newcommand{\mbr}[1]{\left[#1\right]}
\newcommand{\lbr}[1]{\left\{#1\right\}}
\newcommand{\II}{\mathrm{II}}
\renewcommand{\i}{\rm{i}}

\newcommand{\eng}{\mathcal{E}}
\newcommand{\epst}{\texorpdfstring{$\eps$}{epsilon}}
\newcommand{\bM}{\overline{M}}

\newcommand{\driv}[2][u]{\nabla^{#2}#1}

\newcommand{\scomp}{\#}

\newcommand{\pfrac}[2]{\frac{\partial #1}{\partial #2}}

\title[Neck Analysis]{Neck Analysis of Extrinsic Polyharmonic Maps}
\author{Wanjun Ai}
\address{School of Mathematical Sciences, University of Science and Technology of China}%
\email{aiwanjun@mail.ustc.edu.cn}%
\author{Hao Yin}
\address{School of Mathematical Sciences, University of Science and Technology of China}%
\email{haoyin@ustc.edu.cn}%
\thanks{The research work of Hao Yin is supported by NSFC 11471300.}%

\subjclass[2010]{35J60, 35B99}%
\keywords{polyharmonic maps, blow-up analysis, energy identity, no neck}%
\date{\today}

\begin{document}
%\allowdisplaybreaks[4]
\begin{abstract}
  We prove the energy identity and the no neck property for a sequence of smooth extrinsic polyharmonic maps with bounded total energy. 
\end{abstract}
\maketitle
\section{Introduction}\label{sec:intro}
Suppose that $\Omega$ is a smooth domain in $\RB^{n}$ and that $N$ is a smooth compact submanifold in $\RB^K$. For a $W^{2m,2}$-map $u$ from $\Omega$ to $N$, consider the (total) energy
\begin{equation*}
  \eng(u;\Omega) = \frac{1}{2}\int_\Omega|\nabla^{m}u|^2,
\end{equation*}
where $\nabla$ is the usual derivative of a function defined on $\Omega$. $\eng(u;\Omega)$ is called the \emph{extrinsic} $m$-polyenergy, since it depends on the particular embedding of $N$ in $\RB^{K}$.

The critical points of $\eng(u;\Omega)$ are called \emph{extrinsic polyharmonic maps}. They are natural generalizations of harmonic maps, which have been extensively studied in the past decades. We expect many results in the theory of harmonic maps generalize to the case of polyharmonic maps, as confirmed by the study of biharmonic maps, which is a special case $(m=2)$ of polyharmonic maps. Of course, the main challenge is that the Euler-Lagrange equation is now an elliptic system of higher order. 

Above the critical dimension ($n>2m$), the regularity problem for general polyharmonic map ($m>2$)  is even harder than the biharmonic map case, due to the lack of a monotonicity formula. Nothing is known except the result of Angelsberg and Pumberger \cite{AngelsbergPumberger2009regularity} which requires an assumption stronger than the energy bound. In this paper, we restrict ourselves to the critical dimension and assume $n=2m$.

In the critical dimension, following the celebrated work of H\'elein \cite{Helein1991Regularite}, we expect the weak solution to be smooth. In the biharmonic map case ($m=2$), this is confirmed by Chang, Wang and Yang \cite{ChangWangYang1999regularity} when the target manifold is sphere (see also \cite{Strzelecki2003biharmonic}) and by Wang \cite{Wang2004Biharmonic} for general target manifold. For general (extrinsic) polyharmonic maps, Goldstein, Strzelecki and Zatorska-Goldstein \cite{GoldsteinStrzeleckiZatorska2009polyharmonic} proved that a weakly polyharmonic map into sphere is smooth. For general target manifold, Moser \cite{Moser2010Regularity} obtained the same result for energy minimizers using a flow method. It is completely solved in full generality by Gastel and Scheven \cite{GastelScheven2009Regularity}. Hence, in the critical dimension, any weakly polyharmonic map is smooth.

In this paper, we are interested in the compactness of a sequence of \emph{extrinsic and smooth} polyharmonic maps with uniformly bounded energy. More precisely, let $u_i:B\to N$ be such a sequence with
\begin{equation*}
  \eng(u_i):=\eng(u_i;B) \leq \Lambda<\infty
\end{equation*}
where $B$ is the unit ball in $\RB^{2m}$.
Since $n=2m$ is the critical dimension, there is an $\varepsilon$-regularity theorem (see \autoref{thm:regularity}) which allows us to study the compactness and the blow-up of the sequence just as in the theory of harmonic maps in dimension $2$ and biharmonic maps in dimension $4$. Since the multi-bubble analysis is by now very standard (see \cite{DingTian1995Energy} or \cite{Parker1996}), we study only the single bubble case, i.e.
\begin{itemize}
  \item $u_i$ converges locally smoothly on $\bar{B}\setminus \set{0}$ to a polyharmonic map $u_\infty$; 
  \item there is a sequence $\lambda_i\to 0$ such that
    \begin{equation*}
      v_i(x)= u_i(\lambda_i x)
    \end{equation*}
    converges in $C^{\infty}_{loc}(\RB^{2m})$ to a polyharmonic map $\omega$;
  \item $\omega$ is the unique bubble.
\end{itemize}

For any small $\delta>0$ and large $R>0$, the behavior of $u_i$ on the neck domain $B_\delta\setminus B_{\lambda_i R}$ is not \emph{a priori} known by the above assumption. The study of $u_i$ in this domain is usually called the neck analysis and consists of (at least) the following well known questions: 

\begin{enumerate}[label=(Q\arabic{*}), ref=Q\arabic{*}]
  \item \label{item:Q1} Is it true that
    \begin{equation*}
      \lim_{\delta\to 0} \lim_{R\to \infty} \lim_{i\to \infty} \int_{B_\delta\setminus B_{\lambda_i R}} \abs{\nabla^{m} u_i}^2 + \abs{\nabla u_i}^{2m}  =0?
    \end{equation*}
    An affirmative answer, which is known as the \emph{energy identity} or \emph{energy quantization}, would imply that there is no energy loss in the neck domain, i.e.,
    \begin{equation}\label{eqn:Q1}
      \lim_{i\to \infty} E(u_i;B) = E(u_\infty;B) + E(\omega;\RB^{2m}),
    \end{equation}
    where 
    \[
      E(u;\Omega)=\int_\Omega|\nabla^m u|^2+|\nabla u|^{2m}.
    \]
  \item \label{item:Q2} Is it true that
    \begin{equation}
      \lim_{\delta\to 0} \lim_{R\to \infty} \lim_{i\to \infty} \osc_{B_\delta\setminus B_{\lambda_i R}} u_i =0? \label{eqn:Q2}
    \end{equation}
    An affirmative answer, which is usually called the \emph{no neck property}, means that the image of $u_\infty$ touches the image of $\omega$.
\end{enumerate}

These questions have been studied extensively for harmonic maps and biharmonic maps. We refer to \citelist{\cite{Jost1991dimensional} \cite{Parker1996} \cite{Qing1995singularities} \cite{DingTian1995Energy} \cite{Wang1996Bubble}} for harmonic maps and \citelist{\cite{Wang2004Remarks} \cite{WangZheng2013Energy} \cite{WangZheng2012Energy} \cite{HornungMoser2012Energy} \cite{Laurain} \cite{LY2016}} for biharmonic maps.

The main result of this paper is to confirm that the energy identity and the no neck property hold for the above mentioned sequence.
\begin{thm}
  \label{thm:main}
  Suppose that $\set{u_i}$ is a sequence of smooth (extrinsic) $m$-polyharmonic maps defined on $B\subset \RB^{2m}$ and that $\eng(u_i)\leq \Lambda$ for some uniform constant $\Lambda>0$. Assume that $u_i$ converges locally smoothly on $B\setminus \set{0}$ to $u_\infty$ and that there is $\lambda_i\to 0$ such that $v_i(x)\eqdef u_i(\lambda_i x)$ converges locally smoothly on $\RB^{2m}$ to $\omega$. If $\omega$ is the only bubble, then \eqref{eqn:Q1} and \eqref{eqn:Q2} holds.
\end{thm}
An easy corollary of \autoref{thm:main} would be
\begin{cor}
  \label{cor:main} For $u_\infty$ and $\omega$ as above, both limits
  \begin{equation*}
    \lim_{\abs{x}\to 0} u_\infty \quad \text{and} \quad \lim_{\abs{x}\to \infty} \omega
  \end{equation*}
  exists.
\end{cor}
\autoref{cor:main} is an obvious consequence of the removable of singularity theorem, which is proved first by \cite{SU1981} in the harmonic map case and is obtained as a consequence of the regularity of weak solution in the critical dimension for biharmonic maps in \cite{Wang2004Biharmonic}. Notice that the regularity of weak solution of polyharmonic maps in the critical dimension is obtain by \cite{GastelScheven2009Regularity}, hence a removable of singularity theorem should be a corollary of it. We also note that the argument in Section 6 of \cite{LY2016} can be used to give another proof of removable of singularity theorem for polyharmonic maps as the authors did for biharmonic maps.

The proof of \autoref{thm:main} follows the general framework in \cite{LY2016}, which was in turn motivated by the original idea of \cite{SU1981} and \cite{QT1997}. As is well known by now, the main challenge of the neck analysis is that the length of the neck, or equivalently, the ratio $\delta/ (\lambda_i R)$, is out of control. The proof consists of two parts. The idea of the first part is to compare the map on the neck domain (using the cylinder coordinates) with its average on the spheres and notice that the difference is some approximate $m$-polyharmonic function, for which a three-circle lemma holds. This lemma then implies that the energy of $u_i$ involving (at least one) derivatives along the spherical direction decays exponentially along the neck (see \eqref{eqn:tangent}) so that when we do integration on the neck, the length of it does not matter. The second part is some far reaching generalization of the classical Pohozaev identity, which was introduced to the study of the neck analysis of harmonic maps, by \cite{LW1998}.

While many arguments in \cite{LY2016} generalizes without much effort to the case of polyharmonic maps, the two most important ingredients in the argument, namely, the three-circle lemma and the Pohozaev type argument needs completely different proofs. It turns out that the generalization from $m=2$ (the biharmonic map case) to $m>2$ exploits some fine structure of the operator $\Delta^m$. In particular, we obtain the following three-circle lemma for $m$-polyharmonic function, which is new to the best of our knowledge.

\begin{thm}[\autoref{lem:3circle}]
  \label{thm:3circle} Let $B$ be the unit ball in $\RB^{2m}$. There exists some constant $L$ depending only on $m$, such that for any nonzero smooth function $u$ on $B\setminus \set{0}$ satisfying
  \begin{equation*}
    \Delta^m u =0 \quad \text{on} \quad B\setminus \set{0}
  \end{equation*}
  and
  \begin{equation*}
    \int_{\partial B_\rho} u =0, \quad\forall \rho\in(0,1),
  \end{equation*}
  the following \emph{three-circle property} holds: for each $i\in \mathbb Z^+$,
  \begin{equation*}
    2F_i(u) < e^{-L} \left( F_{i-1}(u)+ F_{i+1}(u) \right),
  \end{equation*}
  where
  \begin{equation*}
    F_i(u) = \int_{A_i} \frac{u^2}{\abs{x}^{2m}} dx \quad \text{and} \quad A_i = B_{e^{-(i-1)L}} \setminus B_{e^{-iL}}. 
  \end{equation*}
\end{thm}

The special case $m=2$ is proved in \cite{LY2016}. However, the proof there is by brutal force and hence not instructive. Here we emphasize that this three-circle lemma is related to the \emph{weak orthogonality} of a basis of polyharmonic functions on annular domains. 

The paper is organized as follows. In \autoref{sec:framework}, we follow the proof of \cite{LY2016} to give an outline of the proof of \autoref{thm:main} and state a few results which generalizes to the $m$-polyharmonic maps trivially. It shall be clear by the end of that section that the proof of \autoref{thm:main} is complete except for the three circle \autoref{lem:3circle}, which we prove in \autoref{sec:3circle}, and the Pohozaev type argument, which we give in \autoref{sec:pohozaev}. 

\section{A framework of neck analysis}\label{sec:framework}

Suppose $\set{u_i}$ is the sequence in \autoref{thm:main}. Since $N$ is compact, $\set{u_i}$ are uniformly bounded. Together with the uniform bound of $\eng(u_i)$, it implies by interpolation and Sobolev embedding $W^{m,2}\embedto W^{1,2m}$ that
\begin{equation*}
  E(u_i)\eqdef \int_B \abs{\nabla^{m} u_i}^2 + \abs{\nabla u_i}^{2m}
\end{equation*}
is also uniformly bounded so that the $\varepsilon$-regularity theorem (see \autoref{thm:regularity}) can be applied. 

By the assumption of \autoref{thm:main}, $\omega$ is the only bubble and hence by an argument in \cite{DingTian1995Energy}, we have that: for any $\varepsilon>0$, there exist small $\delta>0$ and large $R>0$ such that
\begin{equation}\label{eqn:DT}
  \limsup_{i\to \infty} \int_{B_{2\rho}\setminus B_{\rho/2}}\abs{\nabla^{m} u_i}^2 + \abs{\nabla u_i}^{2m} < \varepsilon
\end{equation}
for any $\rho \in (2\lambda_i R, \delta/2)$. For sufficiently small $\varepsilon$ and any fixed $\rho$ as above, by setting $v(x)= u_i (\rho x)$ and applying \autoref{thm:regularity} to $v$ as a map defined on $B_2\setminus B_{1/2}$, we obtain, for any $l\in \mathbb N$,
\begin{equation}\label{eqn:good}
  \max_{\partial B_\rho}\rho^{l} \abs{\nabla^l u_i}\leq  C \varepsilon^{1/(2m)}.
\end{equation}

The estimate \eqref{eqn:good} is at the borderline of being able to answer \eqref{item:Q1} and \eqref{item:Q2} and what we need is just any slight improvement of \eqref{eqn:good}. More precisely, \autoref{thm:main} follows from
\begin{nclaim}\label{clm:better}
  There exists some $\sigma>0$ such that
  \begin{equation}
    \max_{\partial B_\rho} \rho^l \abs{\nabla^l u_i} \leq C(\sigma,l) \varepsilon^{1/(2m)} \left( \left(\frac{\delta}{\rho}\right)^{-\sigma} + \left(\frac{\rho}{\lambda_i R}\right)^{-\sigma} \right)
    \label{eqn:better}
  \end{equation}
  for $\rho\in (2\lambda_i R, \delta/2)$ and $l\in \mathbb N$.
\end{nclaim}

\begin{rmk}
  \autoref{clm:better} remains equivalent if we replace the left hand side of \eqref{eqn:better} by
  \begin{equation*}
    \max_{\partial B_\rho} \rho^{l-1}\abs{\nabla^{l-1} (\rho \nabla) u_i}.
  \end{equation*}
  To see this, we need to show that the following two quantities
  \begin{equation*}
    \max_{l=1,\ldots,l_0} \max_{\partial B_\rho} \rho^l \abs{\nabla^l u}
  \end{equation*}
  and
  \begin{equation*}
    \max_{l=1,\ldots,l_0} \max_{\partial B_\rho} \rho^{l-1} \abs{\nabla^{l-1} (\rho \nabla) u}.
  \end{equation*}
  can bound each other up to some constant,

  While this is trivial for $l_0=1$, the general case is proved by induction and the proof, which can be simplified by scaling $\partial B_\rho$ to $\partial B_1$, is omitted. %\footnote{I have checked this and it seems that we need no induction here.}
\end{rmk}

\subsection{The radial-tangential decomposition}\label{sub:decompose}
As usual, we would like to decompose the energy into the radial part and the tangent (to the sphere) part. For example, in the harmonic map case (in dimension two), we have
\begin{equation*}
  \abs{\nabla u}^2 = \abs{\partial_\rho u}^2 + \abs{\rho^{-1}\partial_\theta u}^2.
\end{equation*}

However, in the polyharmonic map case, an analog of the above equation is not obvious due to two reasons: first, the \textit{tangent} derivative (for $m>1$) can not be expressed easily by partial derivatives as $\partial_\theta$ does in the case of $S^1$, and second, the energy involves higher order derivatives instead of just the gradient. 

To solve this problem, recall that $S^{2m-1}$ is the homogeneous space $SO(2m)/SO(2m-1)$. Let $\set{\alpha_k}_{k=1}^{m(2m-1)}$ be an orthonormal basis of the Lie algebra $\mathfrak{so}(2m)$, which induces Killing fields $\set{X_k}_{k=1}^{m(2m-1)}$ on $S^{2m-1}$. By \cite{Helein1991Regularity}*{Lemma~2}, there exist $m(2m-1)$ smooth vector fields $Y_1,\cdots,Y_{m(2m-1)}$ on $S^{2m-1}$ such that for any $y\in S^{2m-1}$ and any tangent vector $V\in T_y S^{2m-1}$, we have
\begin{equation*}
  V= \sum_{k=1}^{m(2m-1)} \inner{V,X_k(y)}_{S^{2m-1}} Y_k(y).
\end{equation*}

Next, we extend the above construction to $\RB^{2m}\setminus \set{0}$. For that purpose, we identify $S^{2m-1}$ with $\partial B_1$, the boundary of the unit ball, and notice that there is a natural diffeomorphism between $\partial B_1$ and $\partial B_\rho$ for all $\rho \ne 0$, which we call $\Psi_\rho$. By using $\Psi_\rho$, $X_k$ and $Y_k$ are extended to be vector fields on $\RB^{2m}\setminus \set{0}$, which are tangent to $\partial B_{\abs{x}}$ at $x\in \RB\setminus \set{0}$. With these definitions and noticing the fact that $\Psi_\rho$ pulls back the inner product on $\partial B_1$ to be 
\begin{equation*}
  \frac{1}{\rho^2} \inner{\cdot,\cdot}_{\partial B_\rho},
\end{equation*}
we have
\begin{equation*}
  V = \sum_{k=1}^{m(2m-1)} \frac{1}{\abs{x}^2}\inner{V,X_k(x)} Y_k(x),
\end{equation*}
for any $x\in \RB^{2m}\setminus \set{0}$ and $V \in T_x (\partial B_{\abs{x}})$. Here we have omitted the subscript of the inner product, because the inner product of $\partial B_\rho$ is induced from that of $\RB^{2m}$ and we now take every vector as living in $\RB^{2m}$.

Finally, by adding the radial component, we arrive at the decomposition formula:
\begin{equation}
  V= \inner{V,\partial_\rho} \partial_\rho + \sum_{k=1}^{m(2m-1)} \frac{1}{\abs{x}^2}\inner{V,X_k(x)} Y_k(x)
  \label{eqn:decompose}
\end{equation}
for all $V\in \RB^{2m}$ and $x\in \RB^{2m}\setminus \set{0}$. In particular, if we let $V=\nabla u_i$ and multiply both sides of \eqref{eqn:decompose} by $|x|=\rho$, we find
\begin{equation}\label{eqn:decomp}
  \rho\nabla u_i = (\rho \partial_\rho u) \partial_\rho + \sum_{k=1}^{m(2m-1)}(X_k u_i) \frac{Y_k}{\rho}.
\end{equation}
As vectors in $\RB^{2m}$, $\partial_\rho(x)$ and $\frac{Y_k}{\rho}(x)$ is independent of $\abs{x}$, which implies that for any $l\geq 0$, there is a constant $C_l$ such that
\begin{equation}\label{eqn:goodvector}
  \max_{\pt B_\rho}\abs{\rho^l \nabla^l (\partial_\rho)}\leq C_l,\quad \max_{\pt B_\rho}\abs{\rho^l \nabla^l \left(\frac{Y_k}{\rho}\right)}\leq C_l.
\end{equation}

By \eqref{eqn:goodvector} and the remark after \autoref{clm:better}, we can take $\rho^{l-1}\nabla^{l-1}$ on both sides of \eqref{eqn:decomp} to see that \autoref{clm:better} (and hence \autoref{thm:main}) is a consequence of
\begin{nclaim}
  There exists some $\sigma>0$ such that for any $l\in \mathbb N$ and $\rho\in (2\lambda_i R, \delta/2)$,
  \begin{align}
    \max_{k=1,\ldots,m(2m-1)} \max_{\partial B_\rho} \rho^l \abs{\nabla^{l} (X_k u_i)} &\leq C(\sigma,l) \varepsilon^{1/(2m)} \left( \left(\frac{\delta}{\rho}\right)^{-\sigma} + \left(\frac{\rho}{\lambda_i R}\right)^{-\sigma} \right)
    \label{eqn:tangent1}\\
    \intertext{and}
    \max_{\partial B_\rho} \rho^l \abs{\nabla^l (\rho\partial_\rho u_i)} &\leq C(\sigma,l) \varepsilon^{1/(2m)}\left( \left(\frac{\delta}{\rho}\right)^{-\sigma} + \left(\frac{\rho}{\lambda_i R}\right)^{-\sigma} \right).
    \label{eqn:radial1}
  \end{align}
\end{nclaim}
In the rest of this paper, we often work in cylinder coordinates $(t,\theta)$, where $t=\log \rho$ and $\theta \in S^{2m-1}$, so that \eqref{eqn:tangent1} and \eqref{eqn:radial1} become 
\begin{equation}
  \max_{k=1,\ldots,m(2m-1)} \max_{\partial B_\rho} \rho^l \abs{\nabla^{l} (X_k u_i)} \leq C(\sigma,l) \varepsilon^{1/(2m)} \left( e^{-\sigma(\log \delta - t)} + e^{-\sigma (t-\log \lambda_i R)} \right)
  \label{eqn:tangent}
\end{equation}
and
\begin{equation}
  \max_{\partial B_\rho} \rho^l \abs{\nabla^l (\partial_t u_i)}  \leq C(\sigma,l) \varepsilon^{1/(2m)}\left( e^{-\sigma(\log \delta - t)} + e^{-\sigma (t-\log \lambda_i R)} \right).
  \label{eqn:radial}
\end{equation}

In the next two subsections, we give the proofs of \eqref{eqn:tangent} and \eqref{eqn:radial} respectively. The proofs are almost complete except for a few \emph{key} lemmas that we prove  in \autoref{sec:3circle} and \autoref{sec:pohozaev}. 

\subsection{The decay of tangential derivatives}
We prove \eqref{eqn:tangent} in this subsection. The idea is the same as in the biharmonic map case \cite{LY2016}, except that we now need the three circle lemma for $m$-polyharmonic functions instead of biharmonic functions. Recall that a function $u$ is $m$-polyharmonic if and only if $\Delta^m u=0$.
\begin{lem}[Three Circle Lemma, \autoref{thm:3circle}]\label{lem:3circle}
  Suppose $B\subset\RB^{2m}$ is the unit ball. There exists a constant $L$ depending only on $m$, such that for any non-zero smooth $m$-polyharmonic function $u$ on $B\setminus\set{0}$ satisfying (in the polar coordinates)
  \begin{equation}\label{eq:3circ-lem:condi}
    \int_{\partial B_\rho}u=0 \quad \forall \rho \in (0,1),
  \end{equation}
  the following \emph{three-circle property} holds
  \begin{equation}\label{eq:3circ}
    2F_i(u)< e^{-L}\sbr{F_{i-1}(u)+F_{i+1}(u)},\quad\forall i=1,2,\ldots,
  \end{equation}
  where
  \begin{equation*}
    F_i(u)\eqdef \int_{A_i} \frac{u^2}{\abs{x}^{2m}} dx \quad \text{and} \quad A_i\eqdef B_{e^{-(i-1)L}} \setminus B_{e^{-iL}}.
  \end{equation*}
\end{lem}
The proof will be given in \autoref{sec:3circle}.

\begin{defn}
  A smooth function $u:B_{r_2}\setminus B_{r_1}$ is called an $\eta$-approximate $m$-polyharmonic function if it satisfies
  \begin{eqnarray}
    \label{eqn:approx}
    \Delta^m u &=& a_1\nabla^{2m-1} u + \cdots + a_{2m}u \nonumber \\
               && + \frac{1}{\abs{\partial B_\rho}} \int_{\partial B_\rho} b_1\nabla^{2m-1} u + \cdots + b_{2m}u
  \end{eqnarray}
  where $a_i$ and $b_i$ are smooth functions satisfying
  \begin{equation*}
    \abs{a_j}+ \abs{b_j} \leq \frac{\eta}{\abs{x}^j} \quad \text{on}\quad B_{r_2}\setminus B_{r_1}
  \end{equation*}
  for $j=1,\cdots,2m$.
\end{defn}

The three circle lemma generalizes to the case of $\eta$-approximate $m$-polyharmonic function for small $\eta$.
\begin{thm}[\cite{LY2016}*{Theorem~3.4}]\label{thm:approx}
  There is some constant $\eta_0>0$ such that the following is true. For $L$ given in \autoref{lem:3circle} and $\eta_0$-approximate $m$-polyharmonic function $u$ defined on $A_{i-1}\cup A_i\cup A_{i+1}$ (for any $i\in \mathbb N$) satisfying
  \begin{equation*}
    \int_{\partial B_\rho} u =0 \quad\forall \rho\in [e^{-(i-2)L}, e^{-(i+1)L}],
  \end{equation*}
  we have
  \begin{enumerate}[label=(\alph{*}), ref=\alph{*}]
    \item if $F_{i+1}(u)\leq e^{-L} F_i(u)$, then $F_i(u)\leq e^{-L} F_{i-1}(u)$;
    \item if $F_{i-1}(u)\leq e^{-L} F_i(u)$, then $F_i(u)\leq e^{-L} F_{i+1}(u)$;
    \item either $F_i(u)\leq e^{-L} F_{i-1}(u)$, or $F_i(u)\leq e^{-L} F_{i+1}(u)$.
  \end{enumerate}
\end{thm}
The proof of this theorem is almost the same as Theorem 3.4 in \cite{LY2016} and is therefore omitted.

To make use of \autoref{thm:approx}, we define
\begin{equation}\label{eqn:wi}
  w_i \eqdef u_i- u^*_i
\end{equation}
where $u^*_i(\rho) \eqdef \dashint_{\partial B_\rho} u_i$ is the average of $u_i$ on $\partial B_\rho$.
\begin{lem}[\cite{LY2016}*{Lemma~4.1}]\label{lem:41}
  There is $\varepsilon_1>0$ such that if $\varepsilon$ in \eqref{eqn:DT} is smaller than $\varepsilon_1$, then $w_i$ defined in \eqref{eqn:wi} is $\eta_0$-approximate $m$-polyharmonic function for the $\eta_0$ given in \autoref{thm:approx}.
\end{lem}

As a corollary, we have
\begin{cor}[\cite{LY2016}*{Lemma~4.2}]\label{cor:42} 		
  For $\varepsilon<\varepsilon_1$ in \eqref{eqn:DT} and $i$ sufficiently large, assume without loss of generality that
  \begin{equation*}
    \Sigma\eqdef B_\delta\setminus B_{\lambda_i R} = \bigcup_{l=l_0}^{l_i} A_l.
  \end{equation*}

  Then we have
  \begin{equation*}
    F_l(w_i)\leq C\varepsilon^{1/m} \left( e^{-L(l-l_0)} + e^{-L(l_i-l)} \right)
  \end{equation*}
  for some constant $C$ independent of $i$.
\end{cor}
The proofs of \autoref{lem:41} and \autoref{cor:42} are straightforward generalization of the corresponding proofs in \cite{LY2016} and are omitted.

The next lemma concludes the discussion of this subsection by giving a proof of \eqref{eqn:tangent}, which is stronger than Lemma 4.3 of \cite{LY2016} in the sense that we bound any higher order derivatives. 
\begin{lem} The inequality \eqref{eqn:tangent} holds for $\sigma=1/2$.
  \label{lem:43}
\end{lem}
\begin{proof}
  For each $\rho=e^t \in [2\lambda_i R, \frac{\delta}{2}]$, \autoref{cor:42} implies that
  \begin{equation*}
    \int_{B_{2\rho}\setminus B_{\rho/2}} \frac{w_i^2}{\abs{x}^{2m}} dx \leq C \varepsilon^{1/m} \left( e^{-(t-\log (\lambda_i R))} + e^{-(\log \delta -t)} \right).
  \end{equation*}
  Recall that by \autoref{lem:41}, $w_i$ is a solution of \eqref{eqn:approx}. A version of $L^p$ estimate applied to \eqref{eqn:approx} (see \cite{LY2016}*{Lemma~3.3}) gives
  \begin{equation}\label{eqn:Xkui}
    \max_{k=1,\ldots,m(2m-1)} \max_{B_{3\rho/2}\setminus B_{2\rho/3}} \abs{X_k u_i} \leq C \varepsilon^{1/(2m)} \left( e^{-\frac{1}{2}(t-\log (\lambda_i R))} + e^{-\frac{1}{2}(\log \delta -t)} \right).
  \end{equation}
  Here $X_k$ are the Killing fields in \autoref{sub:decompose}. All higher order estimates of $X_k u_i$ follows from  \eqref{eqn:Xkui} and \autoref{lem:high}.
\end{proof}
\subsection{The decay of radial derivatives}\label{sub:pohozaev}
In the case of harmonic maps (in dimension two), the well known Pohozaev identity ensures that
\begin{equation*}
  \int_{S^1} \rho^{-2} \abs{\pfrac{u_i}{\theta}}^2 d\theta = \int_{S^1} \abs{\pfrac{u_i}{\rho}}^2 d\theta.
\end{equation*}
The proof of this equality is simply the integration by parts and the elementary observation that
\begin{equation*}
  \int_{S^1} \Delta u_i \cdot \partial_\rho u_i d\theta=0.
\end{equation*}

In \cite{LY2016}, the authors discovered that by a similar argument, one obtains not the direct comparison between the tangential and radial energy, but an ordinary differential inequality which describes the dynamics of the radial energy along the neck, and in which the tangential energy appears as a small perturbation term. In other words, the radial energy itself has an internal tendency of exponential decay, rather than simply following the decay of tangential energy. This is a key difference between the neck analysis of harmonic maps and polyharmonic maps.

In the rest of this subsection, we present this argument in the setting of (extrinsic) polyharmonic maps and show that a similar first order non-homogeneous linear difference inequality (c.f.~\eqref{eq:ode}) still gives the decay of radial energy. The derivation of this ordinary differential inequality exploits some special structure of the polyharmonic operator $\Delta^m$, the details of which is proved in \autoref{sec:pohozaev}. 

We prefer working in the cylinder coordinates $(t,\theta)$, where $\rho=e^t$ and $\theta\in S^{2m-1}$. Since
\begin{equation*}
  \Delta u(t,\theta) = e^{-2t} \left( \partial_t^2 u + 2(m-1)\partial_t u + \Delta_{S^{2m-1}} u \right)
\end{equation*}
and
\begin{equation*}
  [e^{-\alpha t}, \Delta_{S^{2m-1}}]=0, \qquad[e^{-\alpha t}, \partial_t] = \alpha e^{-\alpha t}
\end{equation*}
for any $\alpha\in \RB$,
there is a polynomial $P$ of constant coefficients such that 
\begin{equation}\label{eqn:P}
  \Delta^m u = e^{-2mt} P(\partial_t, \Delta_{S^{2m-1}}) u.
\end{equation}

Our Pohozaev type argument relies crucially on the following elementary observations on $P$, whose proof can be found in \autoref{sec:pohozaev}.
\begin{lem}
  \label{lem:P} Suppose that the polynomial $P$ defined in \eqref{eqn:P} is given by
  \begin{equation}\label{eq:polynomial_P}
    P(\partial_t,\Delta_{S^{2m-1}})= \sum_{\stackrel{1\leq p+2q\leq 2m}{p,q\in \mathbb N\bigcup\set{0}}} a_{p,q} \partial_t^p \Delta_{S^{2m-1}}^q.
  \end{equation}
  Then the constant coefficients $\set{a_{p,q}}$ satisfy 
  \begin{equation*}
    a_{p,q}=0 \quad \text{for } p=1,3,\cdots,2m-1,
  \end{equation*}
  and
  \begin{equation}\label{eq:coeffi}
    (-1)^{m-\frac{p}{2}} a_{p,0}>0\quad \text{for } p=2,4,\cdots,2m.
  \end{equation}
\end{lem}

As usual, the starting point of the Pohozaev argument is 
\begin{equation}\label{eqn:start}
  \int_{_{S^{2m-1}}} \Delta^m u \cdot \partial_t u d\theta =0,
\end{equation}
where we have used the fact that $\partial_t u$ is a tangent vector of $N$ at $u$ and $\Delta^m u$ is a normal vector of $N$ at the same point by the Euler-Lagrange equation of the extrinsic polyharmonic maps (see \cite{AngelsbergPumberger2009regularity}*{Lemma~2.1}).

By using integration by parts and \eqref{eqn:P}, \eqref{eq:polynomial_P}, \eqref{eqn:start}, we obtain
\begin{lem}
  \label{lem:first} Assume that $u$ is a smooth (extrinsic) $m$-polyharmonic map defined on $B\setminus \set{0}$. Then
  \begin{equation}\label{eq:pohozaev:EL}
    \int_{S^{2m-1}} P(\partial_t, \Delta_{S^{2m-1}}) u \cdot \partial_t u = \partial_ t \int_{S^{2m-1}} (Q+\Theta)=0,
  \end{equation}
  where
  \begin{equation}\label{eq:pohozaev:normal}
    \begin{split}
      Q&=\sum_{n=1}^m(-1)^{n+1}\sbr{n-1/2}a_{2n,0}|\pt_t^{n}u|^2\\
       &\qquad\qquad+\pt_t\mbr{\sum_{n=1}^m\sum_{k=1}^{n-1}\sum_{l=1}^{n-k}(-1)^{k+l}a_{2n,0}\pt_t^{k+l-1}u\pt_t^{2n-k-l}u
      },
    \end{split}
  \end{equation}
  and $\Theta$ is a quadratic homogeneous polynomial of
  \begin{equation*}
    \partial_t^p (\nabla_{S^{2m-1}})^{q_2}(\Delta_{S^{2m-1}})^{q_1}  u
  \end{equation*}
  where $q_2$ is either $0$ or $1$ and $1\leq 2 q_1+q_2\leq 2m-1$, $p=0,1,\ldots,2m-3$.
\end{lem}
The proof of \autoref{lem:first} will be given in \autoref{sec:pohozaev}.

Now, for any fixed $i$, since $u_i$ as a function defined on $B$ is smooth at the origin,  
\[
  \lim_{t\to-\infty}(Q+\Theta)=0,
\]
which implies that
\begin{equation}
  \lim_{t\to-\infty}\int_{S^{2m-1}}(Q+\Theta) d\theta=0.
  \label{eqn:zero}
\end{equation}
By integrating \eqref{eq:pohozaev:EL} over $(-\infty,t)$ and using \eqref{eqn:zero}, we know that for all $t<0$, 
\begin{equation}\label{eq:pohozaev:alternating}
  \begin{split}
    0&=\int_{S^{2m-1}}\Bigg\{\Theta+\sum_{n=1}^m\sbr{n-1/2}(-1)^{n+1}a_{2n,0}|\pt_t^nu_i|^2\\
     &\qquad\qquad\qquad+\pt_t\mbr{\sum_{n=1}^m\sum_{k=1}^{n-1}\sum_{l=1}^{n-k}(-1)^{k+l}a_{2n,0}\pt_t^{k+l-1}u_i\pt_t^{2n-k-l}u_i}\Bigg\}.
    \end{split}
  \end{equation}

  This equation \eqref{eq:pohozaev:alternating} is the polyharmonic map version of Pohozaev identity. The first term involves the $\theta$-derivatives of $u_i$ and decays as we need by \autoref{lem:43}. More precisely, we have
  \begin{equation}\label{eqn:thetadecay}
    \abs{\Theta}\leq C(\sigma) \varepsilon^{1/m} \left(  e^{-\sigma(\log \delta -t)} + e^{-\sigma(t-\log(\lambda_i R))} \right)
  \end{equation}
  for some $\sigma>0$ and any $t \in (\log (2 \lambda_i R), \log(\delta/2))$. To see this, we observe that
  \begin{equation*}
    \nabla_{S^{2m-1}} u_i = \rho \nabla_{\partial B_{\rho}} u_i = \sum_{k=1}^{m(2m-1)} (X_k u_i) \frac{Y_k}{\rho} \qquad \text{by \eqref{eqn:decomp}}
  \end{equation*}
  and
  \begin{equation*}
    \abs{\partial_t^p \nabla^q_{S^{2m-1}} w}\leq C \max_{l=1,\ldots,p+q} \abs{\rho^l \nabla^l w} \qquad \text{for any function } w.
  \end{equation*}
  Hence, \eqref{eqn:thetadecay} follows from \eqref{eqn:tangent} and \eqref{eqn:goodvector}

  Thanks to \autoref{lem:P}, the second term of \eqref{eq:pohozaev:alternating} (up to a sign) dominates
  \begin{equation*}
    \int_{S^{2m-1}} \sum_{n=1}^m \abs{\partial_t^n u_i}^2,
  \end{equation*}
  which is a sum of the radial derivatives. However, \eqref{eq:pohozaev:alternating} does not give a direct comparison between $\Theta$ and the radial derivatives (represented by the second term) because of the existence of the last term.

  Before we prove \eqref{eqn:radial}, we integrate \eqref{eq:pohozaev:alternating} over the whole neck cylinder $S^{2m-1}\times (\log \lambda_i R, \log \delta)$ to see
  \begin{multline*}
    \int_{\log \lambda_i R}^{\log \delta} \int_{S^{2m-1}} \sum_{n=1}^m \abs{\pt_t^{n}u_i}^2 d\theta dt
    \leq C \int_{\log \lambda_i R}^{\log \delta} \int_{S^{2m-1}} \abs{\Theta} d\theta dt\\
    + C\sum_{n=1}^m \sum_{k=1}^{n-1}\sum_{l=1}^{n-k} \left\{ \int_{S^{2m-1}\times \set{\log \lambda_i R}} + \int_{S^{2m-1}\times \set{\log \delta}} \right\} \abs{a_{2n,0} \pt_t^{k+l-1}u_i \pt_t^{2n-k-l}u_i},
  \end{multline*}
  for some constant $C$ depends only on $m$. By \eqref{eqn:thetadecay} and \eqref{eqn:good}, we have
  \begin{equation}
    \int_{\log \lambda_i R}^{\log \delta} \int_{S^{2m-1}} \sum_{n=1}^m \abs{\pt_t^{n}u_i}^2 d\theta dt \leq C \varepsilon^{1/m}.
    \label{eqn:totalsmall}
  \end{equation}
  \begin{rmk}
    Although we do not need this for the proof of our main theorem, we remark that \eqref{eqn:totalsmall} and \eqref{eqn:tangent} together could be used to give an affirmative answer to \eqref{item:Q1} in the introduction. It suffices to get an expression of the energy density in cylinder coordinates and show that the integration of its radial part on the neck domain is small by \eqref{eqn:totalsmall}.
  \end{rmk}

  For the proof of \eqref{eqn:radial}, we notice that \eqref{eq:pohozaev:alternating} amounts to a difference inequality, which gives the desired exponential decay in \eqref{eqn:radial}. To see this, consider some fixed $t_0\in[\log\lambda_i R,\log\delta]$ and for each $t\in[1,\min\set{t_0-\log\lambda_i R,\log\delta-t_0}-1]$, set
  \[
    F(t)=\int_{t_0-t}^{t_0+t}\int_{S^{2m-1}}\sum_{n=1}^m(n-1/2)(-1)^{m-n}a_{2n,0}|\pt_t^{n}u_i|^2d\theta dt.
  \]
  Multiplying by $(-1)^{m+1}$ and integrating \eqref{eq:pohozaev:alternating} over $(t_0-t,t_0+t)$ yields
  \begin{multline*}
    F(t)=\int_{t_0-t}^{t_0+t}\int_{S^{2m-1}}\lbr{(-1)^m\Theta+\pt_t\mbr{\sum_{n=1}^m\sum_{k=1}^{n-1}\sum_{l=1}^{n-k}(-1)^{m+k+l}a_{2n,0}\pt_t^{k+l-1}u_i\pt_t^{2n-k-l}u_i}}\\
    \leq\int_{t_0-t}^{t_0+t}\int_{S^{2m-1}}|\Theta|+\sum_{n=1}^m\sum_{k=1}^{n-1}\sum_{l=1}^{n-k}\lbr{\int_{S^{2m-1}\times\set{t_0-t}}+\int_{S^{2m-1}\times\set{t_0+t}}}|a_{2n,0}\pt_t^{k+l-1}u_i\pt_t^{2n-k-l}u_i|.
  \end{multline*}
  Since $\partial_t u_i$ satisfies some homogeneous elliptic system whose coefficients are well controlled, we have (see \autoref{lem:high} and notice that $k+l-1,2n-k-l\geq 1$)
  \begin{equation}
    |\pt_t^{k+l-1}u_i\pt_t^{2n-k-l}u_i|_{S^{2m-1}\times\set{t}}\leq C\|\pt_tu_i\|_{2;S^{2m-1}\times[t-1,t+1]}^2,
    \label{eq:control_higher_driv_t}
  \end{equation}
  which implies that 
  \begin{equation}
    \begin{split}
      F(t)&\leq\int_{t_0-t}^{t_0+t}\int_{S^{2m-1}}|\Theta|
      +C\sbr{\int_{t_0-t-1}^{t_0-t+1}\int_{S^{2m-1}}+\int_{t_0+t-1}^{t_0+t+1}\int_{S^{2m-1}}}|\pt_tu_i|^2\\
      &\leq\int_{t_0-t}^{t_0+t}\int_{S^{2m-1}}|\Theta|+C\sbr{\int_{t_0-t-1}^{t_0+t+1}
      \int_{S^{2m-1}}-\int_{t_0-t+1}^{t_0+t-1}\int_{S^{2m-1}}}|\pt_tu_i|^2\\
      &\leq\int_{t_0-t}^{t_0+t}\int_{S^{2m-1}}|\Theta|+C\sbr{F(t+1)-F(t-1)}.
    \end{split}
    \label{eq:ode}
  \end{equation}

  Our next lemma makes use of \eqref{eq:ode} to prove the exponential decay that we need.
  \begin{lem}\label{lem:pohozaev}
    Let $t_0$ and $F(t)$ be as above, then 
    \begin{equation}
      F(1)\leq C\eps^{1/m}\sbr{e^{-\tilde{\sigma}(\log\delta-t_0)}+e^{-\tilde{\sigma}(t_0-\log\lambda_i R)}},
      \label{eq:F1_decay}
    \end{equation}
    where $C$ and $\tilde{\sigma}$ are two constants independent of $u_i$ and $\eps$ is the constant in \eqref{eqn:DT}.
  \end{lem}
  We postpone the proof of \autoref{lem:pohozaev} to \autoref{sec:pohozaev}. Assuming \autoref{lem:pohozaev}, \eqref{eqn:radial} follows from \eqref{eq:F1_decay} and \autoref{lem:high}.

  \section{Three circle lemma for \texorpdfstring{$m$}{m}-polyharmonic function}\label{sec:3circle}
  In this section, we prove \autoref{lem:3circle}, which is a property of $m$-polyharmonic function. By a well known argument of separation of variables, we can write down all $m$-polyharmonic functions defined on an annular domain.

  \begin{lem}
    \label{lem:polyharmonic} Suppose that $u$ is an $m$-polyharmonic function defined on $B_{r_2}\setminus B_{r_1} \subset \RB^{2m}$ and let $(t,\theta)$ be the cylinder coordinates as before. Then there exist real numbers $A_0,B_0, C_{0,k}, D_{0,k}$ and $C_{n,k}^l, D_{n,k}^l$ such that
    \begin{align*}
      u(t,\theta)&=A_0+B_0t+\sum_{k=2}^m\sbr{C_{0,k}e^{2(k-1)t}+D_{0,k}e^{-2(k-1)t}}\\
                 &\qquad+\sum_{n=1}^\infty\sum_{l=1}^{h_n}\sum_{k=1}^m\sbr{C_{n,k}^le^{(n+2(k-1))t}
      +D_{n,k}^le^{-(n+2(k-1))t}}\phi_n^l(\theta).
    \end{align*}
    Here $\set{\phi_n^l|\, n\in \mathbb N \text{ and } l=1,\cdots,h_n}$ is the set of normalized spherical harmonic functions such that
    \[
      \Delta_{S^{2m-1}}\phi_n^l=-n(n+2m-2)\phi_n^l
    \] 
    and they form an orthonormal basis of $L^2(S^{2m-1})$. 

    Moreover, if for each $\rho\in (r_1,r_2)$
    \begin{equation*}
      \int_{\partial B_\rho} u =0,
    \end{equation*}
    then $A_0=B_0=C_{0,k}=D_{0,k}=0$ in the above expansion.
  \end{lem}

  For the proof of \autoref{lem:3circle}, we first observe that by the scaling invariance it suffices to prove the lemma for $i=1$. Then we compute $F_i(u)$ in cylinder coordinates, using \autoref{lem:polyharmonic}, 
  \begin{align*}
    F_i(u)&=\int_{e^{-iL}}^{e^{-(i-1)L}}\int_{S^{2m-1}}\frac{u^2}{\rho^{2m}}\rho^{2m-1}d\rho d\theta
    =\int_{-iL}^{-iL+L}\int_{S^{2m-1}}u^2dtd\theta\\
    &=\int_{-iL}^{-iL+L}\sum_{n=1}^\infty\sum_{l=1}^{h_n}\mbr{\sum_{k=1}^m\sbr{C_{n,k}^le^{(n+2(k-1))t}
    +D_{n,k}^le^{-(n+2(k-1))t}}}^2dt.
  \end{align*}
  Here in the above computation, we exploit the fact that $\set{\phi_n^l}$ is an orthonormal basis. As a consequence of the above computation, it suffices to prove \autoref{lem:3circle} for 
  \begin{equation*}
    u(t,\theta) = \sum_{k=1}^m \left( A_k e^{(n+2(k-1))t} + B_k e^{-(n+2(k-1))t} \right) \phi^l_n(\theta)
  \end{equation*}
  with fixed $l$ and $n$.

  For simplicity, set
  \begin{equation*}
    n_k= n+2(k-1),\quad k=1,2,\ldots,m
  \end{equation*}
  and
  \begin{equation*}
    G(t)= \left( \sum_{k=1}^m (A_k e^{n_k t}+ B_k e^{-n_k t}) \right)^2.
  \end{equation*}
  \autoref{lem:3circle} is now reduced to the claim that there exists some universal number $L>0$ such that
  \begin{equation}\label{eqn:reduce}
    2 \int_{-L}^0 G(t) dt < e^{-L} \int_{-L}^0 G(t-L) + G(t+L) dt
  \end{equation}
  holds for any $A_k$ and $B_k$. Since $A_k$ and $B_k$ are arbitrary, \eqref{eqn:reduce} can be replaced by
  \begin{equation}
    2 \int_{-L/2}^{L/2} G(t) dt < e^{-L} \int_{-L/2}^{L/2} G(t-L) + G(t+L) dt\eqdef e^{-L}(\I + \II)
    \label{eqn:final}
  \end{equation}
  By the elementary inequality
  \begin{equation*}
    (x+y)^2 \geq \frac{x^2}{2}-y^2, \quad  \forall x,y\in \mathbb R,
  \end{equation*}
  we compute
  \begin{align*}
    \I&=\int_{-3L/2}^{-L/2}\mbr{\sum_{k=1}^m\sbr{A_ke^{n_kt}+B_ke^{-n_kt}}}^2dt\\
      &=\int_{L/2}^{3L/2}\mbr{\sum_{k=1}^m\sbr{A_ke^{-n_kt}+B_ke^{n_kt}}}^2dt\\
      &\geq\frac{1}{2}\int_{L/2}^{3L/2}\sbr{\sum_{k=1}^mB_ke^{n_kt}}^2dt
    -\int_{L/2}^{3L/2}\sbr{\sum_{k=1}^mA_ke^{-n_kt}}^2dt\eqdef\frac{\I_1}{2}-\I_2,
  \end{align*}
  and similarly
  \begin{align*}
    \II&=\int_{L/2}^{3L/2}\mbr{\sum_{k=1}^m\sbr{A_ke^{n_kt}+B_ke^{-n_kt}}}^2dt\\
       &\geq\frac{1}{2}\int_{L/2}^{3L/2}\sbr{\sum_{k=1}^mA_ke^{n_kt}}^2dt
    -\int_{L/2}^{3L/2}\sbr{\sum_{k=1}^mB_ke^{-n_kt}}^2dt\eqdef\frac{\II_1}{2}-\II_2.
  \end{align*}

  We need the following lemma
  \begin{lem}\label{lem:claim}
    There exists $L_0>0$ and $\tilde{C}$ depending only on $m$ such that for any $x_0\in \RB$, $A_k\in \RB$, $L>L_0$ and $n\in \mathbb N$,
    \[
      \int_{x_0}^{x_0+L}\sbr{\sum_{k=1}^mA_ke^{n_kt}}^2dt
      \geq\delta_{m,n}\int_{x_0}^{x_0+L}\sum_{k=1}^mA_k^2e^{2n_kt}dt,
    \]
    where $\delta_{m,n}=\tilde{C} e^{-nL/2}$.
  \end{lem}
  This result is the ``weak orthogonal'' relation mentioned in the introduction and we have an explicit formula for $\delta_{m,n}$, from which we learn that the ``orthogonal'' relation gets weaker and weaker as $n$ gets larger. Fortunately, this is enough for the proof of \autoref{lem:3circle}. Before the proof of \autoref{lem:claim}, we show how \autoref{lem:3circle} follows from it.

  By \autoref{lem:claim} and $n_k\geq n$, 
  \begin{equation}
    \I_1 + \II_1 \geq\delta_{m,n}\int_{L/2}^{3L/2}\sum_{k=1}^m (A_k^2 + B_k^2) e^{2n_kt}dt \geq \delta_{m,n}Le^{nL} \sum_{k=1}^m (A_k^2+B_k^2).
    \label{eqn:one}
  \end{equation}
  By Schwartz inequality and $n_k\geq n$ again,
  \[
    \I_2 + \II_2\leq m\int_{L/2}^{3L/2}\sum_{k=1}^m( A_k^2 + B_k^2) e^{-2n_kt}dt\leq m Le^{-nL} \sum_{k=1}^m (A_k^2+B_k^2).
  \]
  There is $L_1>L_0$ (depending only on $m$) such that for all $L>L_1$, we have
  \[
    \delta_{m,n}e^{nL}= \tilde{C} e^{nL/2} \geq 4me^{-nL},\quad\forall n\geq1,
  \]
  which implies that
  \begin{equation*}
    \I_1+ \II_1 \geq 4 (\I_2 +\II_2)
  \end{equation*}
  and
  \begin{equation}
    \I+\II\geq\frac{1}{2}\sbr{\I_1+\II_1}-\I_2-\II_2\geq\frac{1}{4}\sbr{\I_1+\II_1}.
    \label{eqn:quater}
  \end{equation}

  Back to the proof of \eqref{eqn:final}, its right hand side is estimated by
  \begin{eqnarray*}
    \mathrm{RHS} &\geq& e^{-L}\sbr{\I+\II}\geq\frac{e^{-L}}{4}\sbr{\I_1+\II_1} \\
                 &\geq&\frac{\delta_{m,n}e^{-L}}{4}\int_{L/2}^{3L/2}\sum_{k=1}^m\sbr{A_k^2+B_k^2}e^{2n_kt}dt \\
                 &=& \sum_{k=1}^m \left( \frac{\delta_{m,n} e^{-L}}{8n_k} (e^{3n_k L}- e^{n_k L})\right) \cdot  (A_k^2+ B_k^2),
  \end{eqnarray*}
  where we have used \eqref{eqn:quater} and the first half of \eqref{eqn:one}.  
  On the other hand, by Schwartz inequality, the left hand side of \eqref{eqn:final} is estimated by
  \begin{eqnarray*}
    \mathrm{LHS} &\leq& 4m\int_{-L/2}^{L/2}\sum_{k=1}^m\sbr{A_k^2e^{2n_kt}+B_k^2e^{-2n_kt}}dt \\
                 &\leq& \sum_{k=1}^m 4m L e^{n_k L} (A_k^2 + B_k^2)
  \end{eqnarray*}
  Since $\sum_{k=1}^m A_k^2+B_k^2\ne 0$, for the proof of \eqref{eqn:final}, it suffices to show 
  \[
    4me^{n_k L}L< \frac{\delta_{m,n}e^{-L}}{8n_k}(e^{3n_kL}-e^{n_kL}), \quad \forall k=1,\cdots,m.
  \] 
  Using $\delta_{m,n}=\tilde{C}e^{-nL/2}$ and $n\leq n_k <n+2m$, it is further reduced to
  \begin{equation}
    32m(n+2m)L< \tilde Ce^{-(n/2+1)L}\sbr{e^{2n L}-1}.
    \label{eqn:done}
  \end{equation}

  In summary, to prove \autoref{lem:3circle}, it suffices to find $L$ such that \eqref{eqn:done} holds for any $n$. For that purpose, we choose $L>L_1$ satisfying
  \begin{enumerate}
    \item \label{item:clone_i}
      \begin{equation*}
        e^{2nL}-1 > \frac{1}{2}e^{2nL},\quad \forall n\in \mathbb N;
      \end{equation*}	
    \item \label{item:clone_iii}
      \begin{equation*}
        32 m (nL) \leq \frac{\tilde{C}}{4} e^{nL/2},\quad \forall n\in \mathbb N;
      \end{equation*}
    \item \label{item:clone_ii}
      \begin{equation*}
        64m^2 L \leq \frac{\tilde{C}}{4} e^{L/2}.
      \end{equation*}
  \end{enumerate}
  \eqref{item:clone_i} implies that the right hand side of \eqref{eqn:done} is strictly larger than 
  \begin{equation*}
    \frac{\tilde{C}}{2} e^{3nL/2-L} \geq \frac{\tilde{C}}{2} e^{nL/2} \geq \frac{\tilde{C}}{4}\left( e^{nL/2}+e^{L/2} \right).
  \end{equation*}
  On the other hand, by (ii) and (iii), the left hand side of \eqref{eqn:done} is no larger than
  \begin{equation*}
    \frac{\tilde{C}}{4}(e^{nL/2}+e^{L/2}).
  \end{equation*}
  Hence we finish the proof of \autoref{lem:3circle}.

  %Basically, the above Lemma states that the exponential decay of angles between
  \vskip 2mm
  The rest of this section is devoted to the proof of \autoref{lem:claim}. By choosing a different set of $A_k$ and setting $t'=t-x_0$ if necessary, we may assume that $x_0=0$. Setting
  \begin{equation*}
    f_k(t) = \frac{e^{n_k t}}{ \left( \int_0^L e^{2n_k t} dt \right)^{1/2}}
  \end{equation*}
  and exploiting the arbitrariness of $A_k$ again, \autoref{lem:claim} is reduced to the inequality
  \begin{equation}
    \int_{0}^{L}\sbr{\sum_{k=1}^mA_kf_k(t)}^2dt\geq\delta_{m,n}\int_{0}^{L}\sum_{k=1}^mA_k^2f_k^2(t)dt=\delta_{m,n}\sum_{k=1}^mA_k^2.
    \label{eq:3circ:claim:reduction}
  \end{equation}

  Consider $m\times m$ matrix $M$, the entry of which in the $k$-th row and $j$-th column is given by
  \[
    M_{kl}\eqdef\inner{f_k,f_l}=\int_{0}^{L}f_k(t)f_l(t)dt.
  \]
  Clearly, $M$ is positive definite as $\set{f_k}_{k=1}^m$ are linearly independent. Denote the (real) eigenvalues of $M$ by $\lambda_m\geq\lambda_{m-1}\geq\cdots\geq\lambda_1> 0$. It suffices to show that $\lambda_1\geq\delta_{m,n}$, which is equivalent to \eqref{eq:3circ:claim:reduction}. Direct computation shows that
  \[
    M_{kl}=\frac{2\sqrt{n_kn_l}}{n_k+n_l}\frac{1-e^{-(n_k+n_l)L}}{\sqrt{1-e^{-2n_kL}}\sqrt{1-e^{-2n_lL}}}.
  \]
  The lower bound of $\lambda_1$ is obtained by taking $M$ as the perturbation of $\bM$ with error $E$, i.e., $M=\bM+ E$, where
  \[
    \bM_{kl}=2\frac{\sqrt{n_kn_l}}{n_k+n_l}.
  \]
  $\bM$ is also positive definite. An easy way to see this is to notice that
  \begin{equation*}
    \bM_{kl}= \int_0^{+\infty} \tilde{f}_k \tilde{f}_l dt
  \end{equation*}
  where
  \begin{equation*}
    \tilde{f}_k(t)= \frac{e^{n_k t}}{\left( \int_0^{+\infty} e^{2n_k t} dt \right)^{1/2}},
  \end{equation*}
  and that $\set{\tilde{f}_k}_{k=1}^m$ are linearly independent on $(0,+\infty)$.

  The smallest eigenvalue $\bar\lambda_1$ of $\bM$ is related to $\lambda_1$ by
  \begin{align*}
    \lambda_1&=\inf_{|X|=1}X^TMX=\inf_{|X|=1}\sbr{X^T\bM X+X^TEX}\\
             &\geq\bar\lambda_1-\sup_{|X|=1}|X^TEX|\geq\bar\lambda_1-|\lambda_E|,
  \end{align*}
  where $|\lambda_E|=\sup_{|X|=1}|X^TEX|$. Hence, a lower bound of $\lambda_1$ follows from a lower bound of $\bar{\lambda}_1$ and an upper bound of $\abs{\lambda_E}$.

  Note that $E$ is symmetric, thus for some orthogonal matrix $P$, $P^TEP=\Lambda$ is diagonal. One can estimate $\lambda_E$ by 
  \begin{align*}
    |\lambda_E|^2&=\sup_{|Y|=1}|Y^T\Lambda Y|^2\leq\tr(\Lambda^2)=\tr(E^2)\leq m^2\max_{k,l}|E_{kl}|^2\\
                 &=m^2\max_{k,l}\abs{\frac{2\sqrt{n_kn_l}}{n_k+n_l} \sbr{\frac{1-e^{-(n_k+n_l)L}}{\sqrt{1-e^{-2n_kL}}\sqrt{1-e^{-2n_lL}}}-1}}^2\\
                 &\leq m^2\max_{k,l} \sbr{\frac{1-e^{-(n_k+n_l)L}}{\sqrt{1-e^{-2n_kL}}\sqrt{1-e^{-2n_lL}}}-1}^2\\
                 &\leq m^2\sbr{\frac{1}{1-e^{-2nL}}-1}^2=m^2\sbr{\frac{e^{-2nL}}{1-e^{-2nL}}}^2\leq 4m^2e^{-4nL},\quad\forall L\geq1,
  \end{align*}
  where we have used (in the forth line above) the facts
  \[
    n_k\geq n\geq1\xtext{and}\frac{1-xy}{\sqrt{1-x^2}\sqrt{1-y^2}}-1\geq0,\quad\forall x,y\in(0,1).
  \]
  Thus
  \begin{equation}
    |\lambda_E|\leq2me^{-2nL},\quad\forall L\geq1.
    \label{eq:3circ:lambda_E}
  \end{equation}

  Next, we will estimate $\bar{\lambda}_1$. Clearly, $\tr(\bM)=m$, thus
  \[
    0<\bar\lambda_1\leq\bar\lambda_2\leq\cdots\leq \bar\lambda_m<m.
  \]
  If we regard $m$ as a fixed number, then $\det \bM$ is a positive function of $n$ alone, which is a rational function (with integer coefficients) of $n$ with the degree of denominator no greater than (very roughly) $m^2$. (To see that $\det \bM$ is a rational function, it suffices to notice that in the definition of determinant, $\sqrt{n_k}$ appears twice: once in the entry in the $k$-th row, once in the entry in the $k$-th column.) Thus for any $L\geq1$, there is $C>0$ depending on $m$, such that 
  \[
    e^{nL/2}\det\bM\geq e^{n/2}\det\bM\geq C>0,\quad \forall n\in \mathbb N.
  \]

  Then $\bar\lambda_1$ can be estimated as
  \[
    \bar\lambda_1=\frac{\det\bM}{\bar\lambda_2\cdots\bar\lambda_m}\geq Ce^{-nL/2}m^{-(m-1)}\eqdef 2\tilde Ce^{-nL/2}.
  \]
  Combining the above estimate with \eqref{eq:3circ:lambda_E}, we obtain
  \[
    \lambda_1\geq\bar\lambda_1-|\lambda_E|
    \geq2\tilde Ce^{-nL/2}\sbr{1-\frac{m}{\tilde C}e^{-3nL/2}}
    \geq2\tilde Ce^{-nL/2}\sbr{1-\frac{m}{\tilde C}e^{-3L/2}}.
  \]
  Thus if we take $L$ large enough (depending only on $m$) such that
  \[
    1-\frac{m}{\tilde C}e^{-3L/2}\geq\frac{1}{2},
  \]
  then
  \[
    \lambda_1\geq\tilde C e^{-nL/2},
  \]
  and we finish the proof of \autoref{lem:claim}.

  \section{The Pohozaev argument}\label{sec:pohozaev}
  In this section, we complete the proof outlined in \autoref{sub:pohozaev}. Precisely, we give details of the proofs of \autoref{lem:P}, \autoref{lem:first} and \autoref{lem:pohozaev}. Each proof takes up a subsection.

  \subsection{The structure of the polynomial \texorpdfstring{$P$}{P}}
  Recall that we have defined $P$ to be the polynomial satisfying
  \begin{equation*}
    \Delta^m u = e^{-2mt} P(\partial_t, \Delta_{S^{2m-1}}) u.
  \end{equation*}
  The structure of $P$ as described in \autoref{lem:P} is revealed by an induction process. For that purpose, we define
  \begin{equation*}
    \Delta^l u = e^{-2lt} P_l(\partial_t,\Delta_{S^{2m-1}})u
  \end{equation*}
  for $l=1,2,\cdots,m$ with $P=P_m$ and 
  \begin{equation*}
    P_1(\partial_t,\Delta_{S^{2m-1}})= \partial_t^2 + 2(m-1)\partial_t + \Delta_{S^{2m-1}}= \partial_t (\partial_t + 2(m-1)) + \Delta_{S^{2m-1}}.
  \end{equation*}
  By definition, 
  \begin{eqnarray*}
    P_{l+1} u &=&  e^{2(l+1) t} \Delta^{l+1} u \\
              &=& e^{2(l+1) t} \Delta \left( e^{-2l t} P_lu \right) \\
              &=& e^{2l t} (\partial_t^2 + 2(m-1)\partial_t + \Delta_{S^{2m-1}}) \left( e^{-2l t} P_lu \right).
  \end{eqnarray*}
  Direct computation shows
  \begin{equation}\label{eqn:plone}
    P_{l+1}u = \left(\Delta_{S^{2m-1}}+ \partial_t^2+ 2(m-1-2l)\partial_t - 4l(m-1-l) \right) P_lu.
  \end{equation}
  If we set
  \begin{equation*}
    \square_l = \partial_t^2 + \Delta_{S^{2m-1}} - 4l(m-1-l)
  \end{equation*}
  and
  \begin{equation*}
    b_l=2(m-1-2l),
  \end{equation*}
  then \eqref{eqn:plone} becomes
  \begin{equation}\label{eqn:pl2}
    P_{l+1} = (\square_l + b_l \partial_t) P_l
  \end{equation}
  for $l=0,1,\cdots,m-1$, if we define $P_0$ by $P_0u=u$. 

  Using \eqref{eqn:pl2}, we obtain
  \begin{equation}\label{eqn:P1}
    P(u)=P_m(u)=(\square_{m-1} + b_{m-1} \partial_t)\cdots(\square_0+b_0\pt_t) u.
  \end{equation}
  It is elementary to check that
  \begin{enumerate}
    \item $\square_l$ and $b_k\partial_t$ are commutative;
    \item $\square_l=\square_{m-1-l}$;
    \item $b_l = -b_{m-1-l}$. In particular, if $m-1=2l$, then $b_l=b_{m-1-l}=0$.
  \end{enumerate}
  Using these observations, we can rewrite \eqref{eqn:P1} according to the parity of $m$: 
  \begin{itemize}
    \item if $m=2l$ is even, then
      \begin{equation*}
        P(u)=\left( \square_l^2 - b_l^2 \partial_t^2 \right)\cdots \left( \square_0^2 - b_0^2 \partial_t^2 \right);
      \end{equation*}
    \item if $m=2l+1$ is odd, then
      \begin{equation*}
        P(u)=\square_{l}\left( \square_{l-1}^2 - b_{l-1}^2 \partial_t^2 \right)\cdots\left( \square_0^2 - b_0^2 \partial_t^2 \right).
      \end{equation*}
  \end{itemize}
  In either case, we know that there is no odd power of $\partial_t$ in $P$, which is the first claim in \autoref{lem:P}.

  To show the alternating property of $a_{p,0}$, we note that 
  \begin{align*}
    P(\pt_t,0)&=\prod_{l=0}^{m-1}\sbr{\pt_t^2-4l(m-1-l)+2(m-1-2l)\pt_t}\\
              &=\prod_{l=0}^{m-1}\sbr{\pt_t-2l}\sbr{\pt_t+2(m-1-l)}\\
              &=\prod_{l=0}^{m-1}\sbr{\pt_t-2l}  \prod_{j=0}^{m-1} \sbr{\pt_t+2j}\\
    \qquad&=\pt_t^2(\pt_t^2-2^2)\cdots(\pt_t^2-(2(m-1))^2),
  \end{align*}
  which clearly implies \eqref{eq:coeffi}.
  \subsection{The relative Pohozaev identity}
  In this subsection, we use \autoref{lem:P} to prove \autoref{lem:first}. Before we start the proof, let's explain the reason why \autoref{lem:first} is called ``the relative Pohozaev identity''. There is a harmonic map version of \autoref{lem:first}, which can be proved by simple integration by parts. It says that for a harmonic map $u$ defined on an annular domain of $\RB^2$, we have (in cylinder coordinates again)
  \begin{equation}\label{eqn:relative}
    \partial_t \int_{S^1} \abs{\partial_t u}^2 - \abs{\partial_\theta u}^2 d\theta =0.
  \end{equation}
  Instead of asserting the equality of the tangential and the radial energy, it implies that the difference is independent of $t$. Of course, in every interesting application, with reasonable assumptions, we know that
  \begin{equation*}
    \lim_{t\to -\infty} \int_{S^1} \abs{\partial_t u}^2 - \abs{\partial_\theta u}^2 d\theta =0.
  \end{equation*}
  Hence, integrating \eqref{eqn:relative} over $t$ gives the usual Pohozaev identity. That's why we call \eqref{eqn:relative} the relative Pohozaev identity. Obviously, \autoref{lem:first} is the polyharmonic map version of \eqref{eqn:relative}. 

  For the proof of \autoref{lem:first}, we compute (using \eqref{eq:polynomial_P}),
  \begin{equation*}
    Q_{p,q}\eqdef \int_{S^{2m-1}} a_{p,q} \partial_t^p \Delta_{S^{2m-1}}^q u \cdot \partial_t u d\theta.
  \end{equation*}
  By integration by parts, 
  \begin{equation*}
    Q_{p,q} = \int_{S^{2m-1}} (-1)^q a_{p,q} \partial_t^p w_q \cdot \partial_t w_q d\theta
  \end{equation*}
  for
  \begin{equation*}
    w_q =\left\{
      \begin{array}[]{ll}
        \Delta_{S^{2m-1}}^l u & q=2l \\
        \nabla_{S^{2m-1}} \Delta_{S^{2m-1}}^{l} u & q =2l+1.
      \end{array}
    \right.
  \end{equation*}

  Before we proceed, we list two elementary formulas that can be checked by direct computation, 
  \begin{itemize}
    \item for all $n\in \mathbb N$,
      \begin{align}\label{eqn:1}
        \pt_t^{2n}w \cdot \pt_t^{1}w &= \partial_t \left( (-1)^{n+1} \frac{1}{2} \abs{\pt_t^{n}w}^2 + \sum_{k=1}^{n-1} (-1)^{k+1} \pt_t^{k}w \cdot \pt_t^{2n-k}w \right); 	 
      \end{align}
    \item for all $n\in \mathbb N$ and $1\leq k\leq n-1$,
      \begin{align}
        \label{eqn:2}
        \pt_t^{k}w \cdot \pt_t^{2n-k}w &= \partial_t \left( \sum_{l=1}^{n-k} (-1)^{l-1} \pt_t^{k+l-1}w \cdot \pt_t^{2n-k-l}w \right) + (-1)^{n-k} \abs{\pt_t^{n}w}^2.
      \end{align}
  \end{itemize}

  Using \eqref{eqn:1} and the first assertion of \autoref{lem:P}, 
  \begin{multline*}
    \int_{S^{2m-1}} P(u) \cdot \partial_t u d\theta = \sum_{\stackrel{1\leq p+2q\leq 2m}{p,q\in \mathbb N\cup\set{0}}} \int_{S^{2m-1}} (-1)^q a_{p,q} \pt_t^{p}w_q \cdot \pt_t w_q  d\theta \\
    \begin{aligned}
      &= \partial_t \left[  \sum_{\stackrel{1\leq n+q\leq m}{n,q\in \mathbb N\cup\set{0}}}
    (-1)^q a_{2n,q} \int_{S^{2m-1}} (-1)^{n+1} \frac{1}{2} \abs{\pt_t^{n}w_q}^2 + \sum_{k=1}^{n-1} (-1)^{k+1} \pt_t^{k}w_q \cdot \pt_t^{2n-k}w_q d\theta \right] \\
    &= \partial_t \left[ \sum_{n=1}^m a_{2n,0}\int_{S^{2m-1}}
  (-1)^{n+1} \frac{1}{2} \abs{\pt_t^{n}u}^2 + \sum_{k=1}^{n-1} (-1)^{k+1} \pt_t^ku \cdot \pt_t^{2n-k}u d\theta + \Theta \right].
\end{aligned}
\end{multline*}
The required form of $Q$ given in \eqref{eq:pohozaev:normal} is obtained by using \eqref{eqn:2} to expand the $\partial_t^k u \cdot \partial_t^{2n-k} u$ in the above line. This concludes the proof of \autoref{lem:first}.
\subsection{The exponential decay}
The starting point of the proof of \autoref{lem:pohozaev} is \eqref{eq:ode}. For simplicity, we assume that $\min\set{\log \delta -t_0, t_0-\log (\lambda_i R)}$ is an integer, which we denote by $n_0$. By setting
\begin{equation*}
  F_n \eqdef F(n),\quad\forall n=0,1,\ldots,n_0,
\end{equation*}
and
\begin{equation*}
  \Theta_n\eqdef \int_{t_0-n}^{t_0+n} \int_{S^{2m-1}} \abs{\Theta} d\theta ds,
\end{equation*}
\eqref{eq:ode} becomes
\begin{equation}\label{eqn:iteration}
  F_n\leq \Theta_n + C_1 (F_{n+1}-F_{n-1}).
\end{equation}

Integrating \eqref{eqn:thetadecay} over $(t_0-n,t_0+n)$ implies that
\begin{equation*}
  \Theta_n \leq C \varepsilon^{1/m} e^{\sigma n}\left( e^{-\sigma (\log \delta -t_0)} + e^{-\sigma(t_0-\log(\lambda_i R))} \right),
\end{equation*}
or equivalently 
\begin{equation}\label{eqn:Theta}
  \Theta_n \leq C_2 \varepsilon^{1/m} e^{-\sigma (n_0-n)}.
\end{equation}

As a consequence of \eqref{eqn:iteration} and \eqref{eqn:Theta}, we have
\begin{lem}\label{lem:good}
  Let $\set{F_n,\Theta_n}_{n=0}^{n_0}$ be some non-negative numbers satisfying \eqref{eqn:Theta} for $n=0,1,\cdots,n_0$ and \eqref{eqn:iteration} for $n=1,\cdots,n_0-1$. If $F_n$ is non-decreasing with respect to $n$, then there exist constants $\tilde{\sigma}>0$ and $C'>0$ such that
  \begin{equation*}
    F_n\leq C' \varepsilon^{1/m} e^{ -\tilde{\sigma} (n_0-n)}, \qquad \forall n=0,1,\cdots,n_0.
  \end{equation*}
\end{lem}
\begin{proof}
  Define 
  \begin{equation*}
    \mathcal{B}= \set{n=1,\cdots, n_0-1|\, C_1(F_{n+1}-F_{n-1})>\Theta_n } \quad \text{and} \quad \mathcal{B}'= \set{1,\cdots,n_0-1}\setminus \mathcal{B}
  \end{equation*}
  for $C_1$ in \eqref{eqn:iteration}.

  If $n\in \mathcal{B}'$, then by \eqref{eqn:iteration} and \eqref{eqn:Theta},
  \begin{equation}\label{eqn:inA}
    F_n\leq 2\Theta_n \leq 2C_2 \varepsilon^{1/m} e^{-\sigma(n_0-n)},
  \end{equation}
  which finishes the proof of the lemma with $C'=2C_2$ and $\tilde\sigma=\sigma$.

  If $n\in \mathcal{B}$, let $n_1$ be the smallest number in $\mathcal{B}'$ which is larger than $n$ (if no such $n_1$ exists, let $n_1$ be $n_0$). If $n_1\ne n_0$, we have
  \begin{equation}\label{eqn:Fn1}
    F_{n_1} \leq 2C_2 \varepsilon^{1/m} e^{-\sigma (n_0-n_1)},
  \end{equation} 
  by \eqref{eqn:inA}. If $n_1=n_0$, \eqref{eqn:Fn1} also holds (for maybe another $C_2$) due to \eqref{eqn:totalsmall}.

  We may assume $n<n_1-1$, otherwise $n=n_1-1$ and the lemma follows from the monotonicity of $F_n$ and \eqref{eqn:Fn1} by letting $C'=2C_2e^\sigma$ and $\tilde\sigma=\sigma$. Let $l$ be the largest odd number such that
  \begin{equation*}
    n+1, n+3, \ldots, n+l < n_1.
  \end{equation*}
  By the definition of $l$ and $n_1$,
  \begin{align} \label{eqn:dl}
    &n+l+1\leq n_1\leq n+l+2\\ 
    \label{eqn:dn1}
    & n+1,n+3,\ldots,n+l \in \mathcal{B}.
  \end{align}
  By \eqref{eqn:dn1} and \eqref{eqn:iteration},
  \begin{eqnarray*}
    F_{n+1}&\leq& 2C_1(F_{n+2}-F_n) \\
    F_{n+3}&\leq& 2C_1(F_{n+4}-F_{n+2})\\
           &\vdots& \\
    F_{n+l}&\leq& 2C_1(F_{n+l+1}-F_{n+l-1}).
  \end{eqnarray*}
  Using the monotonicity of $F_n$, we may replace $F_{n+1}$ by $F_n$, \ldots, $F_{n+l}$ by $F_{n+l-1}$ in the left hand side of the above inequalities and rewrite them into the form
  \begin{eqnarray*}
    F_n  & \leq &\frac{2C_1}{2C_1+1} F_{n+2} \\
    F_{n+2}& \leq &\frac{2C_1}{2C_1+1} F_{n+4}\\
           &\vdots& \\
    F_{n+l-1}&\leq& \frac{2C_1}{2C_1+1} F_{n+l+1}.
  \end{eqnarray*}
  Hence, multiplying all the above inequalities gives
  \begin{equation*}
    F_n\leq \left( \frac{2C_1}{2C_1+1} \right)^{\frac{l+1}{2}} F_{n+l+1}.
  \end{equation*}
  Noticing $n+l+1\leq n_1$ (by \eqref{eqn:dl}) and using \eqref{eqn:Fn1}, we get
  \begin{equation*}
    F_n\leq \left( \frac{2C_1}{2C_1+1} \right)^{\frac{l+1}{2}} 2C_2\varepsilon^{1/m} e^{-\sigma(n_0-n_1)}.
  \end{equation*}
  Setting $\sigma'=-\frac{1}{2}\log \frac{2C_1}{2C_1+1}$ and using $l+1\geq n_1-n-1$ (by \eqref{eqn:dl}), we obtain 
  \begin{equation*}
    F_n \leq 2C_2 e^{-\sigma'} \eps^{1/m} e^{-\sigma'(n_1-n)}e^{-\sigma(n_0-n_1)}.
  \end{equation*}
  The lemma follows by taking $\tilde{\sigma}=\min \set{\sigma, \sigma'}$ and $C'=2C_2 e^{-\sigma'}$.
\end{proof}

Clearly \autoref{lem:good} implies
\begin{equation*}
  F_1 \leq C'e^{\tilde\sigma}\varepsilon^{1/m} \left( e^{-\tilde{\sigma} (\log \delta -t_0)} + e^{-\tilde{\sigma}(t_0-\log (\lambda_i R))} \right),
\end{equation*}
and we finish the proof of \autoref{lem:pohozaev}.
\appendix
\section{\epst-regularity} 
In this section we prove the $\eps$-regularity theorem for (extrinsic) $m$-polyharmonic maps.
\begin{thm}[$\eps$-regularity]\label{thm:regularity}
  Suppose $u\in W^{2m,p}(B,N)$, $p>1$, is an $m$-polyharmonic map from the unit ball $B\subset\RB^{2m}$ to a compact Riemannian manifold $N$ isometrically embedded  into $\RB^K$. There exists $\eps_0>0$ such that if
  \[
    E(u;B)\eqdef\frac{1}{2}\int_{B}\sbr{|\nabla^mu|^2+|\nabla u|^{2m}}\leq\eps_0,
  \]
  then for some constant $C$,
  \begin{align}
    \label{eq:epsreg:wkp}
    \|u-\bar u\|_{W^{2m,p}(B_{1/2})}&\leq C\sbr{\|\nabla^mu\|_{L^2(B)}+\|\nabla u\|_{L^{2m}(B)}},\\
    \intertext{and for $l=1,2,\ldots$,}
    \label{eq:epsreg:Calpha}
    \|\nabla^lu\|_{L^\infty(B_{1/2})}&\leq C(l)\sbr{\|\nabla^mu\|_{L^2(B)}+\|\nabla u\|_{L^{2m}(B)}},
  \end{align}
  where $\bar u$ is the mean value of $u$ over $B\subset\RB^{2m}$ and $C$, $\set{C(l)}$ are constants.
\end{thm}
\begin{rmk}
  We will show that the Sobolev norm $\|u\|_{W^{m,2}(B)}$ can be controlled by $E(u;B)$ provided that $\bar u=0$. In fact, since $u\in W^{2m,p}(B)\embedto W^{m,2}(B)$ and $N$ is compact, we know that $u-\bar u\in W^{m,2}(B)$. The interpolation inequality (see \cite{AdamsFournier2003Sobolev}*{Lemma~5.2(1)}) applied to $u-\bar u$ gives
  \[
    \|\nabla^ju\|_{L^2(B)}\leq C\sbr{\|\nabla^mu\|_{L^2(B)}+\|u-\bar u\|_{L^{2}(B)}},\quad\forall 1<j\leq m.
  \]
  Combining it with Poincar\'e inequality
  \[
    \|u-\bar u\|_{L^2(B)}\leq C\|\nabla u\|_{L^2(B)},
  \]
  we obtain
  \begin{equation}\label{eq:epsreg_bdd_uj}
    \begin{split}
      \|\nabla^ju\|_{L^2(B)}&\leq C\sbr{\|\nabla^mu\|_{L^2(B)}+\|\nabla u\|_{L^2(B)}}\\
                            &\leq C\sbr{\|\nabla^mu\|_{L^2(B)}+\|\nabla u\|_{L^{2m}(B)}}
      ,\quad\forall 1<j\leq m.
    \end{split}
  \end{equation}

  If in addition, we assume that $\bar u=0$, then the Poincar\'e inequality reads
  \[
    \|u\|_{L^2(B)}\leq C\|\nabla u\|_{L^2(B)},
  \]
  and 
  \[
    \|u\|_{W^{m,2}(B)}^2=\sum_{j=0}^{m}\|\nabla^ju\|_{L^2(B)}^2\leq C\sum_{j=1}^m\|\nabla^ju\|_{L^2(B)}^2.
  \]
  Since \eqref{eq:epsreg_bdd_uj} holds trivially for $j=1$, we conclude that 
  \begin{equation}\label{eq:epsreg_bdd_uwm2}
    \|u\|_{W^{m,2}(B)}\leq C\sbr{\|\nabla^mu\|_{L^2(B)}+\|\nabla u\|_{L^{2m}(B)}},
  \end{equation}
  provided that $\bar u=0$.
\end{rmk}
\begin{proof}
  Without loss of generality, we may assume that $\bar u=0$. Since $u$ is an (extrinsic) $m$-polyharmonic map, it satisfies the following Euler-Lagrange equation (see~\cite{AngelsbergPumberger2009regularity}*{Lemma~2.2, (4)})
  \begin{equation}
    \Delta^mu=\sum_{\substack{i,j,q\geq0\\i+j+q=m-1\\(i,j,q)\neq(0,m-1,0)}}c_{ijq}^{m-1}\Delta^i\nabla^q(P(u))\Delta^{j+1}\nabla^qu
    -\Delta^{m-1}(A(u)(\nabla u,\nabla u)),
    \label{eq:extrinsic_EL}
  \end{equation}
  where $c_{ijq}^m$ are positive integers, $P(u)=D\Pi(u)$ is the orthonormal projection defined by the differential of the nearest projection map $\Pi:N_\delta\to N$, $N_\delta\subset\RB^K$ is a neighborhood of $N$ and $A(u)$ is the second fundamental form of $N\embedto\RB^K$. 

  For a multi-index $\alpha=(k_1,k_2,\ldots, k_a)$, where $k_i\geq1$ are integers for $i=1,2,\ldots,a$, the \emph{norm} of $\alpha$ is defined to be $|\alpha|\eqdef\sum_{i=1}^ak_i$ and $a$ is called the \emph{length} of $\alpha$. We can rewrite \eqref{eq:extrinsic_EL} as
  \begin{equation}\label{eq:EL}
    \Delta^mu=\sum_{\stackrel{|\alpha|=2m}{\alpha\neq(2m)}}a_\alpha(u)\nabla^{k_1}u\scomp\nabla^{k_2}u\scomp\cdots\scomp\nabla^{k_a}u,
  \end{equation}
  where $a_\alpha$ are smooth functions on $N$. Moreover, we may assume that $k_1\geq k_2\geq\cdots k_a\geq1$ in \eqref{eq:EL}. In particular, $1\leq k_1\leq 2m-1$ and $k_2\leq m$.

  Now, for fixed $\sigma\in(1/2,1)$, let $2\sigma'=1+\sigma$ and $\phi\in C_0^\infty(B_{\sigma'})$ be a cutoff function satisfying 
  \[
    \phi|_{B_\sigma}\equiv1\quad \text{and} \quad|\nabla^k\phi|\leq 4^k(1-\sigma)^{-k},\qquad  k=0,1,2,\cdots.
  \]
  In order to apply $L^p$-estimates to $\phi u$, we compute the equation of $\phi u$ as follows
  \begin{align*}
    &\Delta^m(\phi u)-\sum_{i=1}^{2m}\nabla^i\phi\scomp\nabla^{2m-i}u
    =\phi\Delta^mu=\sum_{\stackrel{|\alpha|=2m}{\alpha\neq(2m)}}\phi a_\alpha(u) \nabla^{k_1}u\scomp\nabla^{k_2}u\scomp\cdots\scomp\nabla^{k_a}u
    \\
    &\qquad=\sum_{\stackrel{|\alpha|=2m}{\alpha\neq(2m)}}a_\alpha(u)\mbr{\driv[(\phi u)]{k_1}-\sum_{j=1}^{k_1}\binom{k_1}{j}\driv[\phi]{j}\driv{k_1-j}}
    \scomp\cdots\scomp\driv{k_a}\\	
    &\qquad=\sum_{\stackrel{|\alpha|=2m}{\alpha\neq(2m)}}a_\alpha(u)\driv[(\phi u)]{k_1}\scomp\cdots\scomp\driv{k_a}+\sum_{\stackrel{|\alpha|=2m}{\alpha\neq(2m)}}\sum_{j=1}^{k_1}a_\alpha(u)\driv[\phi]{j}\scomp\driv{k_1-j}\scomp\cdots\scomp\driv{k_a}.
  \end{align*}
  Thus
  \begin{equation}
    \label{eq:EL_phiu}
    \begin{split}
      \Delta^m(\phi u)&=\sum_{i=1}^{2m}\nabla^i\phi\scomp\nabla^{2m-i}u+\sum_{\stackrel{|\alpha|=2m}{\alpha\neq(2m)}}a_\alpha(u)\driv[(\phi u)]{k_1}\scomp\cdots\scomp\driv{k_a}\\
                      &\qquad+\sum_{\stackrel{|\alpha|=2m}{\alpha\neq(2m)}}\sum_{j=1}^{k_1}a_\alpha(u)\driv[\phi]{j}\scomp\driv{k_1-j}\scomp\cdots\scomp\driv{k_a}.
    \end{split}
  \end{equation}

  Next, for $1<p<\frac{2m}{2m-1}$, we estimate the $L^p(B)$ norm of the last three terms in \eqref{eq:EL_phiu} one by one. For the first term, by the definition of $\phi$,
  \[
    \|\nabla^i\phi\scomp\nabla^{2m-i}u\|_{L^p(B)}\leq C(1-\sigma)^{-i}\|\driv{2m-i}\|_{L^p(B_{\sigma'})},\quad 1\leq i\leq 2m.
  \]
  For the second term, by generalized H\"older inequality, 
  \begin{align} \label{eqn:before1}
    \|\driv[(\phi u)]{k_1}\scomp\cdots\scomp\driv{k_a}\|_{L^p(B)}&\leq C\|\driv[(\phi u)]{k_1}\|_{L^{\frac{2mp}{2m-(2m-k_1)p}}(B)}\prod_{i=2}^a\|\driv{k_i}\|_{L^{\frac{2m}{k_i}}(B)}\\ \nonumber
                                                                 &\leq C\|\phi u\|_{W^{2m,p}(B)}\prod_{i=2}^a\|u\|_{W^{m,2}(B)},
  \end{align}
  where in the last inequality, we have used the following Sobolev embedding
  \[
    W^{2m,p}\embedto W^{k_1,2mp/(2m-(2m-k_1)p)}\quad \text{and} \quad W^{m,2}\embedto W^{k,2m/k},\quad1\leq k\leq m.
  \]

  The estimate of the last term is obtained by different methods for different values of $k_1$ and $j$.
  \begin{itemize}
    \item If $j=k_1$, exploiting the boundedness of $u$ and by Sobolev inequalities (see \eqref{eqn:before1}), we obtain 
      \begin{align*}
        &\|\driv[\phi]{j}\scomp\driv{k_1-j}\scomp\driv{k_2}\scomp\cdots\scomp\driv{k_a}\|_{L^p(B)}\\
        &\qquad\leq C(1-\sigma)^{-k_1}\|u\scomp\driv{k_2}\scomp\cdots\scomp\driv{k_a}\|_{L^p(B)}\\
        &\qquad\leq C(1-\sigma)^{-k_1}\|u\|_{W^{m,2}(B)}^{a-1}.
      \end{align*}
    \item If $\max\set{1,k_1-j}\leq j<k_1(<2m)$, then by the definition of $\phi$, 
      \begin{align*}
        &\|\driv[\phi]{j}\scomp\driv{k_1-j}\scomp\driv{k_2}\scomp\cdots\scomp\driv{k_a}\|_{L^{p}(B)}\\
        &\qquad\leq C(1-\sigma)^{-j}\|\driv{k_1-j}\|_{L^{\frac{2m}{k_1-j}}(B_{\sigma'})}\prod_{i=2}^a\|\driv{k_i}\|_{L^{\frac{2m}{k_i}}(B_{\sigma'})}\\
        &\qquad\leq C(1-\sigma)^{-j}\|u\|_{W^{m,2}(B)}^{a}.
      \end{align*}
    \item
      If $1\leq j<k_1-m$, then 
      \begin{align*}
        &\|\driv[\phi]{j}\scomp\driv{k_1-j}\scomp\driv{k_2}\scomp\cdots\scomp\driv{k_a}\|_{L^{p}(B)}\\
        &\qquad\leq C(1-\sigma)^{-j}\|\nabla^{k_1-j}u\|_{L^{\frac{2mp}{2m-(2m-k_1)p}}(B_{\sigma'})}\prod_{i=2}^a\|\nabla^{k_i}u\|_{L^{\frac{2m}{k_i}}(B_{\sigma'})}\\
        &\qquad\leq C(1-\sigma)^{-j}\|u\|_{W^{2m-j,p}(B_{\sigma'})}\|u\|_{W^{m,2}(B)}^{a-1}.
      \end{align*}
      By the interpolation inequality (see~\cite{AdamsFournier2003Sobolev}*{Thm.~5.2(1)})
      \[
        \|u\|_{W^{2m-j,p}(B_{\sigma'})}\leq C\sbr{\|\driv{2m-j}\|_{L^{p}(B_{\sigma'})}+\|u\|_{L^p(B_{\sigma'})}},
      \]
      and the boundedness of $\|u\|_{L^p(B)}$, we conclude that
      \begin{multline*}
        \hspace{3em}\|\driv[\phi]{j}\scomp\driv{k_1-j}\scomp\driv{k_2}\scomp\cdots\scomp\driv{k_a}\|_{L^{p}(B)}\\
        \leq C(1-\sigma)^{-j}\|\driv{2m-j}\|_{L^p(B_{\sigma'})}\|u\|_{W^{m,2}(B)}^{a-1}+C(1-\sigma)^{-j}\|u\|_{W^{m,2}(B)}^{a-1}.
      \end{multline*}
  \end{itemize}

  With the $L^p$ bound of the right hand side of \eqref{eq:EL_phiu} obtained as above, we can apply the $L^p$ estimate to \eqref{eq:EL_phiu} to conclude
  \begin{equation}
    \begin{split}
      \|\driv[(\phi u)]{2m}\|_{L^p(B)}&\leq C\sum_{\stackrel{|\alpha|=2m}{\alpha\neq(2m)}}\Bigg\{
        \sum_{j=\max\set{k_1-m,1}}^{k_1}(1-\sigma)^{-j}\|u\|_{W^{m,2}(B)}^{a-1}\\
        &\hspace{3cm}+\sum_{j=1}^{k_1-m-1}(1-\sigma)^{-j}\|u\|_{W^{m,2}(B)}^{a-1}\\
        &\hspace{3cm}+\sum_{j=1}^{k_1-m-1}(1-\sigma)^{-j}\|\driv{2m-j}\|_{L^p(B_{\sigma'})}\|u\|_{W^{m,2}(B)}^{a-1}
      \Bigg\}\\
      &\qquad+C\sum_{i=1}^{2m}(1-\sigma)^{-i}\|\driv{2m-i}\|_{L^p(B_{\sigma'})}.
    \end{split}
    \label{eqn:split}
  \end{equation}
  Here we have used the small energy condition in \autoref{thm:regularity}, which by \eqref{eq:epsreg_bdd_uwm2} implies that $||u||_{W^{m,2}(B)}$ is small so that the second term in \eqref{eq:EL_phiu} are absorbed into the left hand side.

  \eqref{eqn:split} can be further simplified as
  \begin{equation}
    \|\driv{2m}\|_{L^p(B_{\sigma})}\leq C\sum_{\stackrel{|\alpha|=2m}{\alpha\neq(2m)}}\sum_{j=1}^{2m}\bigg\{
      \|u\|_{W^{m,2}(B)}^{a-1}
      +\sbr{1+\|u\|_{W^{m,2}(B)}^{a-1}}\|\driv{2m-j}\|_{L^p(B_{\sigma'})}
    \bigg\}(1-\sigma)^{-j}.
    \label{eqn:xiaqu}
  \end{equation}

  By setting
  \[
    \Phi_j\eqdef\sup_{1/2\leq\sigma\leq1}(1-\sigma)^j\|\driv{j}\|_{L^p(B_\sigma)}, \quad j=0,1,\ldots,2m,
  \]
  and noting that $1-\sigma=2(1-\sigma')$, \eqref{eqn:xiaqu} implies
  \begin{equation}\label{eq:eps_regularity:Phi}
    \Phi_{2m}\leq C\sum_{\stackrel{|\alpha|=2m}{\alpha\neq(2m)}}\sum_{l=0}^{2m-1}\bigg\{
      \|u\|_{W^{m,2}(B)}^{a-1}
    +\sbr{1+\|u\|_{W^{m,2}(B)}^{a-1}}\Phi_l\bigg\}.
  \end{equation}
  Now, we need the following 
  \begin{claim}[c.f.~\cite{GilbargTrudinger2001Elliptic}*{Thm.~7.27, p.~171ff.}]
    There exists some constant $C_0$, such that for any $\bar\eps\leq C_0$,
    \[
      \Phi_j\leq\bar\eps\Phi_{2m}+C\bar\eps^{-\frac{j}{2m-j}}\Phi_0,\quad\forall 0\leq j\leq 2m-1,
    \]
    holds for some constant $C$ depending on $m$, $j$ and $C_0$. 
  \end{claim}
  In fact, by the definition of $\Phi_j$, for any $s>0$, there exists $\sigma_s\in[1/2,1)$, such that
  \[
    \Phi_j\leq(1-\sigma_s)^j\|\nabla^ju\|_{L^p(B_{\sigma_s})}+s.
  \]
  For any $u\in W^{2m,p}(\Omega)$, by interpolation inequality (see \cite{AdamsFournier2003Sobolev}*{Theorem~5.2(1)}), there exists a constant $C_0>0$, such that, for all $\eps\leq 1$ and any $j=0,1,\ldots,2m-1$,
  \[
    \|\nabla^ju\|_{L^p(\Omega)}\leq C_0\sbr{\eps\|\nabla^{2m}u\|_{L^p(\Omega)}+\eps^{-\frac{j}{2m-j}}\|u\|_{L^p(\Omega)}}.
  \]

  Thus 
  \begin{align*}
    \Phi_j&\leq C_0\sbr{\eps(1-\sigma_s)^j\|\nabla^{2m}u\|_{L^p(B_{\sigma_s})}+\eps^{-\frac{j}{2m-j}}(1-\sigma_s)^j\|u\|_{L^p(B_{\sigma_s})}}+s\\
          &\leq\eps_s\Phi_{2m}+C_0^{1+\frac{j}{2m-j}}(\eps_s)^{-\frac{j}{2m-j}}\Phi_0+s,
  \end{align*}
  where $\eps_s=C_0\eps(1-\sigma_s)^{j-2m}$. 

  Lastly, for any $\bar\eps\leq C_0$ and any fixed $s>0$ one can take $\eps=\bar\eps(1-\sigma_s)^{2m-j}/C_0\leq1$, then $\eps_s=\bar\eps$ and the claim follows by taking $s\to0$.

  Note that \eqref{eq:epsreg_bdd_uwm2} implies that $\|u\|_{W^{m,2}(B)}$ are bounded by $E(u;B)\leq\eps_0$, which is small by assumption. Applying the claim to \eqref{eq:eps_regularity:Phi} with small $\bar\eps$ yields
  \begin{equation}\label{eqn:chazhi}
    \Phi_{2m}\leq C\sbr{\Phi_0+\|u\|_{W^{m,2}(B)}^{a-1}}\leq C\sbr{\|\nabla^mu\|_{L^2(B)}+\|\nabla u\|_{L^{2m}(B)}},
  \end{equation}
  where the last inequality follows from the assumption $\bar u=0$, the smallness of $\eps_0$ and \eqref{eq:epsreg_bdd_uwm2}. By \eqref{eqn:chazhi} and the interpolation inequality again, we conclude that, for $p=\frac{2m+1}{2m}\in\sbr{1,\frac{2m}{2m-1}}$, 
  \begin{equation}\label{eq:epsreg_lp}
    \|u\|_{W^{2m,p}(B_\sigma)}\leq C(1-\sigma)^{-2m}\sbr{\|\nabla^mu\|_{L^2(B)}+\|\nabla u\|_{L^{2m}(B)}}.
  \end{equation}

  We will finish the proof by a bootstrap argument. Let us start with $p_0=\frac{2m+1}{2m}<\frac{2m}{2m-1}$. \eqref{eq:epsreg_lp} implies that \eqref{eq:epsreg:wkp} holds for $p=p_0$.

  By the Sobolev embedding $W^{2m,p_0}\embedto W^{l,\frac{2mp_0}{2m-(2m-l)p_0}}$ and the H\"older inequality,
  \begin{align*}
    \|\driv{k_1}\scomp\driv{k_2}\scomp\cdots\scomp\driv{k_a}\|_{L^{\frac{1}{1-a(1-1/p_0)}}}
    &\leq C\prod_{i=1}^a\|\driv{k_i}\|_{L^{\frac{2mp_0}{2m-(2m-k_i)p_0}}}\\
    &\leq C\|u\|_{W^{2m,p_0}}^a.
  \end{align*}
  Since $a\geq2$, we know that 
  \begin{equation}\label{eq:p1}
    \frac{1}{1-a\left(1-\frac{1}{p_0}\right)}\geq \frac{p_0}{2-p_0}=\frac{2m+1}{2m-1}>p_1\eqdef \frac{2m}{2m-1},
  \end{equation}
  thus
  \begin{equation}\label{eq:RHSlp}
    \|\driv{k_1}\scomp\driv{k_2}\scomp\cdots\scomp\driv{k_a}\|_{L^{p_1}}
    \leq C \|u\|_{W^{2m,p_0}}^a
  \end{equation}
  In summary, $L^{p_1}$ norm of the right hand side of \eqref{eq:EL} is bounded by $\|u\|_{W^{2m,p_0}}$, which implies that \eqref{eq:epsreg:wkp} holds for $p_1$. 

  Next we do the iteration 
  \[
    p_{i+1}=\frac{1}{\frac{1}{p_i}-\frac{1}{2m}}=\frac{2m}{2m-i-1},\quad i=1,2,\ldots,2m-2,
  \]
  and show that \eqref{eq:RHSlp} holds with $p_0$ and $p_1$ replaced by $p_i$ and $p_{i+1}$ respectively. Thus the $L^p$-estimate implies that \eqref{eq:epsreg:wkp} holds for $p_{i+1}$. 

  Note that \eqref{eq:RHSlp} holds for $p_1+\delta_1$, where $\delta_1$ is sufficiently small (see~\eqref{eq:p1}). For technical reasons, we will prove the following: Suppose for all $\alpha$ with $|\alpha|=2m\xtext{and}\alpha\neq(2m)$,
  \begin{equation}\label{eq:epsreg:lp'}
    \|\driv{k_1}\scomp\driv{k_2}\scomp\cdots\scomp\driv{k_a}\|_{L^{p_{i+1}+\delta_{i+1}}}\leq C\|u\|_{W^{2m,p_i+\delta_i}}^a,
  \end{equation}
  holds for $i=l\in\set{1,2,\ldots,2m-3}$ and some $\delta_i>0$, we will show that it also holds for $i=l+1$ and some $\delta_{i+1}$. The exact value of $\set{\delta_i}$ does not matter and we only require that they are positive such that \eqref{eq:w2minfty} holds. Choose $\delta_{l}$ sufficiently small, such that
  \begin{align}\label{eq:w2minfty}
    W^{2m,\bar p_l}&\embedto W^{h,\infty},\quad\forall h=0,1,\ldots,l,\\
    \label{eq:w2mpl}
    W^{2m,\bar p_l}&\embedto W^{h,\frac{2m\bar p_l}{2m-(2m-l)\bar p_l}},\quad\forall h=l+1,\ldots,2m,
  \end{align}
  where $\bar p_l=p_l+\delta_l$. To illustrate the idea, let us assume that the multi-indices $\alpha$ in \eqref{eq:epsreg:lp'} satisfy
  \begin{align*}
    k_1\geq k_2\geq\cdots\geq k_{a-j_1-j_2\cdots -j_l}&>k_{a-j_1-j_2\cdots-j_l+1}=\cdots=k_{a-j_1-j_2\cdots-j_{l-1}}=l\\
                                                      &>\cdots\\
                                                      &>k_{a-j_1-j_2+1}=\cdots =k_{a-j_1}=2\\
                                                      &>k_{a-j_1+1}=\cdots =k_{a}=1.
  \end{align*}
  Setting $j=j_1+j_2\cdots+j_{l}$, by H\"older and Sobolev inequality
  \begin{align*}
    &\|\driv{k_1}\scomp\cdots\scomp\driv{k_a}\|_{L^p}\\
    &\qquad\leq C\|\driv{k_1}\scomp\cdots\scomp\driv{k_{a-j}}\|_{L^p}
    \prod_{h=1}^{j_{l}}\|\driv{k_{a-j+h}}\|_{L^\infty}\cdots
    \prod_{h=1}^{j_1}\|\driv{k_{a-j_1+h}}\|_{L^\infty}\\
    &\qquad\leq C\|u\|_{W^{2m, \bar p_l}}^{j}
    \prod_{h=1}^{a-j}\|\nabla^{k_h}u\|_{L^{\frac{2m\bar p_l}{2m-(2m-k_h)\bar p_l}}},\quad \text{(by \eqref{eq:w2minfty})}\\
    &\qquad\leq C\|u\|_{W^{2m,\bar p_l}}^a,\quad \text{(by \eqref{eq:w2mpl})}
  \end{align*}
  where
  \[
    p\leq\frac{1}{1-(a-j)\left( 1-\frac{1}{\bar p_l} \right)-\frac{j_1+2j_2+\cdots+lj_l}{2m}}
  \]
  which attains its minimum at $j_l=\cdots=j_2=0$, $j_1=1$, $a=2$, i.e.,
  \[
    p\geq \frac{1}{\frac{1}{\bar p_l}-\frac{1}{2m}}\eqdef\bar p_{l+1}>p_{l+1}.
  \]
  Therefore \eqref{eq:epsreg_lp} holds for $i=l+1$.

  As a conclusion of the above iteration, we have shown that \eqref{eq:epsreg:wkp} holds for some $\bar p>2m=p_{2m-1}$. Now for general $p>1$, note that 
  \[
    W^{2m,\bar p}\embedto W^{h,\infty},\quad\forall h=0,1,\ldots,2m-1.
  \]
  Thus
  \begin{align*}
    \|\driv{k_1}\scomp\driv{k_2}\scomp\cdots\scomp\driv{k_a}\|_{L^p}
    &\leq C\|\driv{k_1}\scomp\driv{k_2}\scomp\cdots\scomp\driv{k_a}\|_{L^\infty}
    \leq C\|u\|_{W^{2m,\bar p}}^a,
  \end{align*}
  which implies that \eqref{eq:epsreg:wkp} holds for $p$ and we finish the first part of \autoref{thm:regularity}.

  Finally, \eqref{eq:epsreg:Calpha} follows from \eqref{eq:epsreg:wkp} and the standard bootstrap argument.
\end{proof}
\section{Linearized polyharmonic map equation and some higher order estimates}
\label{sec:high}
In \autoref{sec:3circle} and \autoref{sec:pohozaev}, by using either the three circle lemma (\autoref{lem:3circle}) or the Pohozaev type argument, we have proved the decay of the $L^2$ norm of some derivative of $u_i$ ($X_k u_i$ and $\partial_t u_i$ respectively) along the neck. In this section, we provide a lemma which improves the decay of $L^2$ norm to the decay of pointwise higher order norm. The key to the proof is the observation that $X_k u_i$ and $\partial_t u_i=\rho\partial_\rho u_i$ satisfy a homogeneous linear elliptic system with nice coefficients.

Let $u$ be a smooth polyharmonic map defined on $B_4\setminus B_1$. Recall the Euler-Lagrange equation reads
\begin{equation*}
  \Delta^m u = \sum_{\stackrel{|\alpha|=2m}{\alpha\neq(2m)}} a_\alpha(u) \nabla^{k_1} u \# \cdots \# \nabla^{k_a}u.
\end{equation*}
Moreover, we assume that 
\begin{equation}\label{eqn:DT1}
  \max_{B_4\setminus B_1} \abs{\nabla^l u} \leq 1,\quad\forall l=1,2,\ldots,2m-1,
\end{equation}
which holds for each segment of neck due to the Ding-Tian's reduction and $\eps$-regularity theorem (\autoref{thm:regularity}).

\begin{lem}\label{lem:high}
  Suppose $u$ is a polyharmonic map defined on $B_4\setminus B_1$ satisfying \eqref{eqn:DT1} as above. Then	
  \begin{equation}
    \max_{B_3\setminus B_2} \abs{\nabla^l ( (\rho\partial_\rho) u)}^2 \leq C \int_{B_4\setminus B_1} \abs{\rho\partial_\rho u}^2 dx
    \label{eqn:highradial}
  \end{equation}
  and 
  \begin{equation}
    \max_{B_3\setminus B_2} \abs{\nabla^l ( X_k u)}^2 \leq C \int_{B_4\setminus B_1} \abs{X_k u}^2 dx
    \label{eqn:hightangent}
  \end{equation}
  for $k=1,2,\cdots,m(2m-1)$.
\end{lem}

\begin{proof}
  The proof follows from well known elliptic estimates if we can show that $X_ku$ and $(\rho\partial_\rho) u$ satisfy a nice linear equation so that we can apply linear estimates.

  We make use of the fact that the polyharmonic map equation is invariant under the one parameter group generated by $\rho\partial_\rho$ and $X_k$. More precisely, let $\psi_s$ be such a one parameter group and $u_s= u\circ \psi_s$. If $u$ is $m$-polyharmonic map, then so is $u_s$.
  Therefore, $u_s$ satisfies (c.f.~\eqref{eq:extrinsic_EL})
  \begin{equation*}
    \Delta^m u_s = \sum_{\stackrel{|\alpha|=2m}{\alpha\neq(2m)}} a_\alpha(u_s) \nabla^{k_1} u_s \# \cdots \# \nabla^{k_a} u_s.
  \end{equation*}
  Taking $s$-derivative at $s=0$ and denoting $\frac{du_s}{ds}|_{s=0}$ by $h$ gives
  \begin{eqnarray*}
    \Delta^m h &=& \sum_{\stackrel{|\alpha|=2m}{\alpha\neq(2m)}} D a_\alpha(u) \nabla^{k_1} u \# \cdots \# \nabla^{k_a} u \cdot h \\
               && + \sum_{\stackrel{|\alpha|=2m}{\alpha\neq(2m)}} a_\alpha(u) \left( \nabla^{k_1} h \# \cdots \# \nabla^{k_a} u + \cdots + \nabla^{k_1} u \# \cdots \# \nabla^{k_a}h \right).
  \end{eqnarray*}
  This is a linear elliptic system of $h$, whose coefficients are all good by \eqref{eqn:DT1}.

  Finally, we notice that $h$ can be either $ (\rho\partial_\rho) u$, or $X_ku$ in the above computation, depending on the choice of the one parameter group $\psi_s$.
\end{proof}

\nocite{*}
%\bibliography{neck_analysis_poly}


\end{document}